\documentclass[11pt]{amsart}
\usepackage{geometry} 
\calclayout

\usepackage{pgfplots}
\usepackage{psfrag}
\usepackage{amsmath}
\usepackage{calc}
\usepackage{systeme}
\usepackage{amsfonts}
\usepackage{amsthm}
\usepackage{algorithm2e}
\usepackage[applemac]{inputenc}
\usepackage{bbold}
\usepackage[numbers]{natbib}
\newtheorem{hyp}{Hypothesis}

\numberwithin{equation}{section}

\author{%
  Charles Bertucci $^1$  }

\newtheorem{Theorem}{Theorem}[section]
\newtheorem{Lemma}[Theorem]{Lemma}
\newtheorem{Rem}[Theorem]{Remark}
\newtheorem{Def}[Theorem]{Definition}
\newtheorem{Prop}[Theorem]{Proposition}
\newtheorem{Cor}[Theorem]{Corollary}
\newtheorem{Ex}[Theorem]{Example}

\usepackage[foot]{amsaddr}

\newcommand{\be}{\begin{equation}}
\newcommand{\ee}{\end{equation}}
\newcommand{\ba}{\begin{aligned}}
\newcommand{\ea}{\end{aligned}}

\newcommand{\R}{\mathbb{R}}
\newcommand{\T}{\mathbb{T}}
\newcommand{\K}{\mathbb{L}}

\newcommand{\mptd}{\mathcal{P}(\mathbb{T}^d)}

\newcommand{\mprd}{\mathcal{P}(\mathbb{R}^d)}
\newcommand{\eps}{\epsilon}

\newcommand{\ca}{\color{black}}
\newcommand{\cred}{\color{red}}

\title{Stochastic optimal transport and Hamilton-Jacobi-Bellman equations on the set of probability measures}
\thanks{$^1$ : CMAP, \'Ecole polytechnique, UMR 7641, 91120 Palaiseau, France
}
\date{} 
\begin{document}

\maketitle
\begin{abstract}
We introduce a stochastic version of the optimal transport problem. We provide an analysis by means of the study of the associated Hamilton-Jacobi-Bellman equation, which is set on the set of probability measures. We introduce a new definition of viscosity solutions of this equation, which yields general comparison principles, in particular for cases involving terms modeling stochasticity in the optimal control problem. We are then able to establish results of existence and uniqueness of viscosity solutions of the Hamilton-Jacobi-Bellman equation. These results rely on controllability results for stochastic optimal transport that we also establish.
\end{abstract}
\setcounter{tocdepth}{1}
\tableofcontents
\section*{Introduction}
This paper introduces a stochastic version of the famous problem of optimal transport. We consider a dynamic formulation of the classical problem as in Benamou and Brenier \citep{benamou2000computational} and are interested in the case in which the target measure is described by a stochastic process. This problem is a state-constrained stochastic optimal control problem, which is set on the set of probability measures. We adopt the dynamic programming approach and study the associated Hamilton-Jacobi-Bellman (HJB for short) equation. In particular, we prove a general comparison principle for viscosity solutions of HJB equations on the set of probability measures. Moreover, the HJB equation associated to the stochastic optimal transport problem is associated to a singularity at terminal time which models the state constraint which has to be reached the target once the problem is over.\\
\subsection*{Optimal transport}
The optimal transport problem is one of the most famous problems in applied mathematics. It consists in finding the best way to transport a repartition of mass into another one, given a certain cost functional for the transport. Formulated first by Monge in \citep{monge1781memoire}, it has proven to be a mathematical problem of tremendous difficulty. The theoretical comprehension of this problem is now quite complete and the interest has now shifted onto more practical and numerical problems. More details on optimal transport can be found in the book of Villani \citep{villani2009optimal} or in the one of Santambrogio \citep{santambrogio2015optimal}.

A point of view which has proven to be particularly helpful to attack optimal transport is looking at a dynamic formulation of the problem. In the setting introduced by Benamou and Brenier \citep{benamou2000computational}, a time interval $[0,T]$ is given. The problem consists in transporting a measure $m_0$ into another measure $m_T$ in this time interval $[0,T]$. This formulation is somehow closer to applications as its solutions describe precisely how the mass is going to be transported. Moreover, it naturally leads to the notion of geodesics in certain sets of measures.\\

In this paper we consider an extension of the aforementioned dynamic reformulation of optimal transport, in which the final repartition of mass, or target as we shall call it, is stochastic. More precisely, we shall assume that there is Markovian stochastic process $(\nu_t)_{t \geq 0}$ such that the target repartition of mass is $\nu_T$. This problem is introduced in more details in Section \ref{sec:sot}. Remark that such an optimal transport problem falls in the category of stochastic optimal control problem, in a space of measures, with a terminal state constraint.

Let us insist upon the fact that, from the point of view of applications, this stochastic version of the optimal transport problem is natural and should prove to be of interest. Indeed, in our economy, the transportation of goods usually starts before the exact location of the addresses is known. This is for example the case for the delivery of oil in most ports, as tanker leaving the American continent often change routes depending on the price of crude oil in the major European ports. In some sense, the target measure is in this case given by a supply and demand equilibrium which is in general modeled as a stochastic process as it depends on various unpredictable factors. In a finite state space case, an analogue of this problem was studied in a mean field game framework in Bertucci et al. \citep{bertucci2021master}, we also refer to \citep{bertucci2022singular} for a related problem.

Another example is the maintenance of the storage of most major warehouse of supply networks. Indeed in practice, the different levels of storage are maintained based on projections in a first time and only later on they are adjusted to the actual level of demand. The flux of the projections could \\

\subsection*{Hamilton-Jacobi-Bellman equation on the set of probability measures}
As already mentioned, the stochastic optimal transport problem is a stochastic optimal control problem. Hence, naturally, the following study relies at some point on the study of the associated HJB equation. This partial differential equation (PDE) is set on the space of probability measures. We shall prove a comparison principle for viscosity sub and super solutions of this HJB equation. A particularity of the HJB equation associated to optimal transport problems, stochastic or not, is that it is associated with a singular boundary condition in time, namely because of the constraint that the target has to be reached. We provide an analysis of this singularity in Section \ref{sec:sing}. 

The study of HJB equations in infinite dimensional space is now the subject of a huge and rapidly growing literature. We will try to give an overview of the topic. Usually, the study of HJB equations relies mostly on the notion of viscosity solutions, introduced in finite dimensional spaces in Crandall and Lions \citep{crandall1983viscosity}. We refer to Crandall et al. \citep{crandall1992user} for a complete presentation of viscosity solutions in finite dimensional space. The study of this notion in cases modeling state constraints, and the singular behaviour they can produce, is largely due to Lasry and Lions \citep{lasry1989nonlinear} and Soner \citep{soner1986optimal,soner1986optimal2}.

The case of infinite dimensional equations is much more involved. First studies have been done on Hilbert spaces through the lens of viscosity solutions in Crandall and Lions \citep{crandall1985hamilton,crandall1986hamilton} for first order equations and in Lions \citep{lions1989viscosity} for second order problems, for instance. In the Hilbert case, it is easier to understand the structure of the super-differential of functions. This lead to numerous developments and we refer to Fabbri et al. \citep{fabbri2017stochastic} for a detailed study of second order HJB equations, namely on Hilbert spaces. 

In the 2000s, the study of HJB equations on metric spaces also gained interests and several developments were made. We can cite for instance \citep{feng2009,feng2012,feng2013,feng2014,nakayasu,giga2015,feng2021}, marked mainly by Feng and its co-authors. Motivations for these works seem to come from control problems or large deviations of infinite dimensional systems. In metric spaces, the HJB equation is often written with a Hamiltonian which depends on the local slopes of the functions. These techniques allowed to prove several comparison results. Developments in this direction are now advanced, in particular because of the different comparison principles established. These results rely quite often on the evaluation of the Hamiltonian on the squared metric to a given point. For recent developments on this topic, we refer to Liu et al. \citep{liu2021} for the links between various notions of viscosity solutions on metric spaces and to Conforti et al. \citep{conforti2021hamilton} in which the authors extend the techniques of Tataru for comparison principles to metric spaces, given that the HJB equation involves an evolutional variational inequality. Let us insist upon the fact that the main difficulty in the metric case is often the structure of the super-differential (or of the gradient of a function more simply).

Quite recently,  the study of HJB equations set on the space of probability measures, which is a metric space, has gained a lot of interest. This is namely because of its link with the study of potential Mean Field Games or mean field control.  An important case which we do not address in this paper, is when the PDE characterizing the evolution of the probability measure involves second order terms. In this setting, a highly singular first order term appears in the HJB equation. Note that previously mentioned works, such as \citep{conforti2021hamilton} for instance, are also concerned with such a case. In this setting, upon regularity estimates or a priori information, it is possible to establish a comparison principle , such as in Wu and Zhang \citep{wu2020}. However, such information is in general difficult to obtain. A useful approach has been to consider finite dimensional approximations of such equations and then pass to the limit, see for instance Cosso et al. \citep{cosso2021master} which approximates measures with combination of Dirac masses or Cecchin and Delarue \citep{cecchin2022weak} which uses Fourier decomposition. We also send the interested reader to \citep{soner2022} for an optimal control problem on the space of probability measures. 

A major step in the study of HJB equation on the set of probability measures is the so-called Hilbertian approach, or lifting, introduced by Lions in \citep{lions2007cours}. It is essentially the formulation of an equivalent HJB equation set on an Hilbert space. An approach, which we may call more intrinsic, was developed and used in Gangbo and Swiech \citep{gangbo2015existence}, Marigonda and Quincampoix \citep{marigonda2018mayer} and in Jimenez et al. \citep{jimenez2020optimal}. The links between Lions' Hilbertian approach and this more intrinsic approach is presented (among other things) in Gangbo and Tudorascu \citep{gangbo2019differentiability} and in Jimenez et al. \citep{jimenez2022dynamical}. Several authors have also considered methods relying on finite dimensional approximations of the PDE such as Gangbo et al. \citep{gangbo2021finite}, Mayorga and Swiech \citep{mayorga2023finite}.\\

In this manuscript, we adopt a novel approach to treat HJB equation on the set of probability measures. As we are only concerned with equations not modelling so-called i.i.d. noises, or in other words second order terms in the controlled PDE, we build on the fruitful Lions' Hilbertian approach. We show that this approach can be somehow carried on without using explicitly the Hilbert space. Our approach relies on a new notion of super-differential for functions of a probability measure argument. 

\subsection*{Organisation of the paper}The rest of the paper is organized as follows. In Section \ref{sec:sot}, we introduce the main problem at interest and derive the associated HJB equation. In Section \ref{sec:hjb}, we provide the definition of viscosity solutions we are going to use as well as general comparison principles. We then proceed to establish some estimates on the behaviour of the value of the problem near the time singularity in Section \ref{sec:sing}, providing the well-posedness of the value function. We then show in Section \ref{sec:cont} some continuity estimates on the value function of the problem and show why the value functions are indeed viscosity solutions of the associated HJB equations. We then summarize our analysis in Section \ref{sec:conclusion}.

\subsection*{Notation}
Let us now introduce some notation.
\begin{itemize}
\item The $d$-dimensional torus is denoted by $\T^d$. The set of probability measures on $\T^d$ (resp. $\R^d$) is $\mptd$ (resp. $\mathcal{P}(\R^d)$).
\item Consider a function $\phi : \mptd \to \R$. When it is defined, we note for $µ \in \mptd, x \in \T^d$
\be
\nabla_µ\phi(µ,x) = \lim_{\theta \to 0} \frac{\phi((1-\theta)µ + \theta \delta_x) - \phi(µ)}{\theta}.
\ee
\item We note, if it is defined, $D_{µ}\phi(µ,x) = \nabla_x\nabla_µ\phi(µ,x) \in \R^d.$ The second order derivatives are defined similarly.
\item The image measure of a measure $µ$ by a map $T$ is denoted by $T_{\#}µ$. 
\item The set of couplings between two measures $µ$ and $\nu$ is $\Pi(µ,\nu)$.
\item The notations usc and lsc stand for respectively upper semi continuous and lower semi continuous. The inf (resp. sup) enveloppe $U_*$ (resp. $U^*$) of a locally bounded function $U$ is defined by $U_*(x) = \liminf_{y \to x} U(y)$ (resp. $\limsup_{y \to x} U^*(y)$).
\item The law of a random variable $X$ is denoted by $\mathcal{L}(X)$.
\item Given a $n$-uple $x = (x_1,x_2,...,x_n)$, we denote by $\pi_k (x) = x_k$.
\item The set of $d\times d$ symmetric real matrices is denoted by $S_d(\R)$.
\end{itemize}

\section{From deterministic to stochastic optimal transport}\label{sec:sot}
We introduce here the main mathematical problems at interest in this paper, starting with the well-known case of optimal transport. 

\subsection{Optimal transport}
The problem of optimal transport consists in finding the best way (for a particular criteria) to transport a certain repartition of mass to another repartition of mass. We give a short presentation of this problem and refer to Villani \citep{villani2009optimal} and Santamborgio \citep{santambrogio2015optimal} for more details on this topic. Given $\mu$ and $\nu$ two probability measures on measurable sets $E_1$ and $E_2$, the main question of optimal transport is to find optimal measurable maps $T : E_1 \to E_2$ in the problem
\be\label{ot1}
\inf \{ \tilde{c}(T) | T_{\#}\mu = \nu\},
\ee
for a given real valued cost function $\tilde{c}$. Quite often, this cost is taken of the form
\[
\tilde c(T) = \int_{E_1} c(x,T(x))\mu(dx),
\]
where $c : E_1\times E_2 \to \R$. This problem lead to numerous mathematical developments since the seminal work of Monge. In the previous form, the problem has no minimizer in general. To observe this, it suffices to consider $µ$ a Dirac mass and $\nu$ the Lebesgue measure on some real interval. Indeed in this case the infimum is taken over an empty set. To address this issue, one usually introduces the relaxation of Kantorovich \citep{kantorovich1942translocation}. In this relaxed version, the typical form of the optimal transport problem becomes
\be\label{relaxedot}
\inf \int_{E_1\times E_2} c(x,y)\pi(dx,dy),
\ee
where the infimum is taken over all couplings $\pi$ between $µ$ and $\nu$, that is on probability measures on $E_1\times E_2$ such that for any measurable sets $A \subset E_1, B \subset E_2$, $\pi(A\times E_2) = µ(A)$ and $\pi(E_1\times B) = \nu(B)$.\\

Let us also mention the natural probabilistic interpretation of such a problem. Consider a probabilistic space $(\Omega, \mathcal{A},\mathbb{P})$. The previous relaxation of the optimal transport problem can be expressed as 
\be\label{probaot}
\inf_{(X,Y)}\mathbb{E}[c(X,Y)],
\ee
where the infimum is taken over all the couples $(X,Y)$ of random variables on $(\Omega, \mathcal{A},\mathbb{P})$ such that $\mathcal{L}(X) = µ$ and $\mathcal{L}(Y) = \nu$. Questions of existence, uniqueness and stability of optimal transport maps and optimal couplings have been extensively studied since. \\

In this paper, we are mostly interested in the case $E_1 = E_2 = \T^d$. In this case we note $\Pi(µ,\nu)$ the set of couplings $µ$ and $\nu$ in $\mptd$. When the cost $c$ is chosen as $c(x,y) = |x-y|^k$, the value of the optimal transport problem defines the Wasserstein distances through
\[
W_k(µ,\nu) = \left( \inf_{ \pi \in \Pi(µ,\nu)} \int_{\T^{2d}}|x-y|^k\pi(dx,dy) \right)^{\frac 1k}.
\]
The set of optimal couplings for the case $k = 2$ is denoted by $\Pi^{opt}(µ,\nu)$. 

One of the most useful approach for optimal transport problems has been the reformulation of \eqref{ot1} into
\be\label{benamoubrenier}
\inf_{(\alpha,m)} \int_0^1 \int_{\T^d}L(x, \alpha_t(x)) m_t(dx)dt,
\ee
where $L : \T^d \times \R^d \to \R$ is a certain cost function which is assumed to be bounded form below, and the infimum is taken over all pairs $(\alpha,m)$ such that $m : [0,1] \to \mathcal{P}(\T^d)$ is a continuous map, $\alpha : [0,1]\times \T^d \to \R^d$ is measurable and $(\alpha,m)$ satisfies in the weak sense
\[
\begin{aligned}
\partial_t m& + \text{div}(\alpha m) = 0 \text{ in } (0,1)\times \T^d, \\
&m(0) = \mu, m(1) = \nu.
\end{aligned}
\]
Let us insist on the fact that, in general, such a product $\alpha m$ is not well defined as a distribution, and thus the precise sense in which the previous equality holds has to be defined with care, which we postpone for the moment.

This approach is due to Benamou and Brenier \citep{benamou2000computational} and it allows to interpret the optimal transport problem as a dynamic optimal control problem, with a terminal state constraint, where the controlled state is a measure. As shown in \citep{benamou2000computational}, the optimality conditions of the problem \eqref{benamoubrenier} can be expressed through the following system of PDE
\[
\begin{aligned}\label{mfpp}
-\partial_t& u + H(x,\nabla_x u) = 0, \text{ in } (0,1)\times \T^d\\
\partial_t &m - \text{div}(D_pH(x,\nabla_x u)m) = 0 \text{ in } (0,1)\times \T^d,\\
m&(0) = \mu, m(1) = \nu,
\end{aligned}
\]
where $H(x,\cdot)$ is the Legendre transform of $L(x,\cdot)$, given by
\[
H(x,p) := \sup_{\alpha \in \R^d}\{-L(x,p) - \alpha\cdot p\}.
\]
 Let us remark that in this setting, the fact that the duration of the problem is $1$ does not play any sort of role except for fixing some constants. This last approach is similar to the use of Pontryagin's maximum principle in dynamic optimal control. 
 
 \subsection{Optimal transport through dynamic programming}
We give a more dynamical approach, \`a la Bellman, of the optimal transport problem. The first thing to be said is that in this approach, the time parameter is crucial. This is of course obvious since we are doing dynamic programming. We adopt the convention that the terminal time, i.e. the time at which the target measure has to be reached is $T > 0$.\\

Let us introduce, formally, the value function $U$ of the optimal transport problem, defined on $(0,T)\times \mptd^2$ by
\be\label{defu}
U(t,\mu,\nu) = \inf_{\alpha,m} \int_t^T \int_{\T^d}L( x,\alpha(s,x)) m_s(dx)ds,
\ee
where the infimum is taken over all $(\alpha,m)$ satisfying the same measurability condition as in \eqref{benamoubrenier} and such that, in the weak sense,
\[
\begin{aligned}\label{continuity}
\partial_s& m + \text{div}(\alpha m) = 0 \text{ in } (t,T)\times \T^d,\\
&m(t) = \mu, m(T) = \nu.
\end{aligned}
\]
It is very tempting to analyze such a value function by a dynamic programming approach and the associated HJB equation. The study of HJB equations is now an extensively studied topic and we refer to the introduction for related works. The expression of the dynamic programming principle usually takes the form, for $0 < \delta < T-t$
\be\label{ddp}
U(t,µ,\nu) = \inf_{\alpha,m} \left \{ \int_{t}^{t + \delta} \int_{\T^d}L(x,\alpha(s,x)) m_s(dx)ds + U(t+\delta,m_{t+\delta},\nu)\right\},
\ee
where the infimum is taken over the same set as in \eqref{defu} and $m_{t +\delta}$ is the value of $m$ at time $t +\delta$. To obtain the associated HJB equation, the usual method is to divide by $\delta $ and to let $\delta \to 0$ in \eqref{ddp}, under the assumption that $U$ is smooth. Doing so yields
\[
\begin{aligned}
0 =- \partial&_t U(t,µ,\nu)-\\
& - \lim_{\delta \to 0}\inf_{\alpha,m} \left\{ \frac{1}{\delta}\int_t^{t+\delta} \int_{\T^d}L(s, x,\alpha(s,x)) m_s(dx) +\int_{\T^d} D_µ U(t,µ,\nu,y)\cdot\alpha(s,y) m_s(dy) ds   \right\}.
\end{aligned}
\]
where we have used
\[
\begin{aligned}
\delta^{-1}(U(t + \delta, m_{t+\delta},\nu) &- U(t,µ,\nu)) = \delta^{-1}\bigg(U(t+\delta,m_{t+\delta},\nu) - U(t,m_{t+\delta},\nu)\\
& \quad + \int_{0}^1\int_{\T^d} \nabla_µ U(t,µ + \theta(m_{t+\delta} - µ),\nu,y) (m_{t+\delta} - µ)(dy)d\theta\bigg)\\
&= -\partial_t U(t,m_{t+\delta},\nu) + o(1)+ \\
&+ \delta^{-1}\int_{0}^1\int_t^{t+\delta}\int_{\T^d} D_µ U(t,µ + \theta(m_{t+\delta} - µ),\nu,y)\cdot\alpha(s,y) m_s(dy)dsd\theta\\
&= -\partial_tU(t,µ,\nu) + \delta^{-1}\int_t^{t+\delta}\int_{\T^d} D_µ U(t,µ,\nu,y)\cdot\alpha(s,y) m_s(dy)ds + o(1),
\end{aligned}
\]
and the fact that the $o(1)$ is uniform in $(\alpha,m)$ along minimizing sequences of the infimum. We do not insist too much on this assumption which is, in a lot of situations, immediate to verify given that $L$ grows sufficiently fast with the size of $\alpha$. Moreover, our aim is to derive the HJB equation, not particularly to consider smooth solutions of this PDE.\\

Recalling that $U$ is assumed to be smooth here, we finally arrive at the HJB equation

\be\label{hjb1}
-\partial_t U(t,\mu,\nu) + \mathcal{H}\left(\mu,D_µU\right) = 0 \text{ in } (0,T)\times \mptd^2,
\ee
where the Hamiltonian $\mathcal{H}:  \mptd \times (\T^d \to \R^d)\to \R$ is given by
\[
\mathcal{H}(µ,\xi) = \int_{\T^d}H(x,\xi(x))µ(dx).
\]
Note that in order for this Hamiltonian to be well defined, an integrability assumption has to be made on $x \to H(x,\xi(x))$ with respect to the measure $µ$.
\begin{Rem}
To be precise, we emphasize the fact that, a priori, the Hamiltonian $\mathcal{H}$ also depends on $\nu$ since the infimum is taken over all admissible controls. Indeed we have not yet proven that, given any pair $(\alpha,m)$ defined on the time interval $[0,\delta]$ we can construct an admissible pair on $[0,t]$ which coincides with $(\alpha, m)$ on $[0,\delta]$. This will be the case under a controllability assumption, namely that from any starting measure $µ$, we can always transport $µ$ toward $\nu$ in time $t$ in finite cost. This will be the case for most of the problem we are interested in but we shall give an example in which this assumption is not verified.
\end{Rem}
Moreover, because there is the state constraint at the terminal time $T$ that the state measure $µ$ has to be transported toward $\nu$, we expect that $U$ satisfies 
\[
U(T,\mu,\nu) = \begin{cases} 0 \text{ if } \mu = \nu, \\+ \infty \text{ otherwise}.\end{cases}
\]
This is always satisfied by the value function since, if $µ\ne \nu$, then the set of admissible controls is empty and thus the value infinite. However, as we shall see in Section \ref{sec:sing}, the behaviour of $U$ as $t\to T$ might be of a different nature, depending on the nature of the cost $L$.\\

Clearly, in this standard framework, $\nu$ is fixed and $U$ is simply a function of $t$ and $\mu$. In the next section, we shall see why the addition of what is only a parameter here, is helpful to model more general problems.\\

The approach of studying \eqref{hjb1} seems equivalent to \eqref{ot1}. However the situation is the same as in standard finite dimensional optimal control. For several deterministic problems, the use of Pontryagin's maximum principle is efficient to provide a complete mathematical analysis. But for a larger class of problems, it is more convenient to use the dynamic programming principle and the associated HJB equation, this is in particular true for the stochastic problems that we are going to introduce later on.\\

Moreover, as usual in dynamic programming, if one is given a smooth solution $U$ of \eqref{hjb1}, then a (smooth) closed-loop optimal control $\alpha^*$ in \eqref{defu} can be computed using the derivatives of $U$ by using the formula
\[
\alpha^*(t,µ,x) = -D_pH(x,D_µU(t,µ,x)) \text{ in } (0,T)\times \mptd \times \T^d.
\]

\subsection{Warning on the formulation of the HJB equation}
The formal computation which allowed us to derive \eqref{hjb1} holds under a smoothness assumption on the value function which does not hold in general.

Indeed, if it was the case, then consider the problem of optimal transport which starts at $µ = \delta_x$ for some $x \in \T^d$ when the time to reach $\nu$ is $t>0$. If $U$ is smooth, then an optimal control $\alpha$ is given as a smooth function of time and space. In particular, the induced trajectory, i.e. the unique solution of \eqref{continuity} will stay a Dirac mass at all time. Hence, as soon as the target measure $\nu$ is not a Dirac mass, we have a contradiction.

The PDE theory is used to derive the equations for smooth functions, and then provide weaker notion of solutions. However, we emphasize that the previous derivation might lead to a dangerous interpretation of the problem as it could lead to restrict the set of admissible controls. In our opinion, the analogy is very much in the spirit of Kantorovich's relaxation. If we restrict too much the set of admissible controls, we might be missing the only admissible controls. We shall come back later on this fact, as it bears some importance in the choice of the definition of viscosity solutions we are going to take. 

\begin{Ex}\label{ex:quadratic}
In the case of a quadratic cost of optimal transport, i.e. when $L(x,p) = \frac 12 |p|^2$, the associated HJB equation is given by
\be\label{hjbex}
-\partial_t U + \frac{1}{2} \int_{\T^d}|D_µU(t,µ,\nu,x)|^2µ(dx) = 0 \text{ in } (0,\infty)\times \mptd.
\ee
In this case, the value function $U$ is simply given by 
\[
U(t,µ,\nu) = \frac1{2(T-t)} W_2^2(µ,\nu).
\]
In particular, $U$ is not smooth, see for instance Alfonsi and Jourdain \cite{alfonsi}. We shall explain in Section \ref{sec:hjb} in which sense it is a viscosity solution of \eqref{hjbex}.
\end{Ex}

\subsection{Stochastic optimal transport}
This section introduces the main problem at interest in this paper, namely a stochastic version of \eqref{ot1}. In this general formulation of the optimal transport problem, it may seem unclear what to do if either the cost function or any of the measure is random. Hence, we focus on the formulation \eqref{benamoubrenier}. We work on a fixed filtered probability space $(\Omega, \mathcal{A}, \mathbb{P}, (\mathcal{F}_t)_{t \geq 0})$ which is assumed rich enough to contain independent Brownian motions.

We want to model the optimal transport of a given measure toward a stochastic target, in the time horizon $T > 0$. We assume here that the target measure is represented by an adapted Markovian stochastic process $(\nu_s)_{s \geq 0}$, valued in $\mptd$ and the (stochastic) target is given by $\nu_{T}$. The problem we want to model is the following: the controlled state is a measure, whose value at time $t$ shall be denoted $µ_t$. At time $t$, the trajectory $(\nu_s)_{s \in [0,t]}$ is known (obviously we do not know the future values of the target process, as this would put us in the usual framework). Then, we want to minimize a certain cost while transporting $µ$ toward $\nu_{T}$. 

As usual in stochastic optimal control, some assumptions have to be made on how the optimization problem takes into account the randomness. To simplify the following discussion, we assume that the problem is risk neutral, hence the problem at interest is given by
\be\label{sot}
\inf_{(\alpha,m)} \mathbb{E}\left[ \int_0^{T}\int_{\T^d}L(x,\alpha_s(x))m_s(dx)ds\right],
\ee
where $\alpha : \Omega\times[0,T]\times \T^d\to \R^d$ and $m:\Omega\times [0,T] \to \mptd$ have to be measurable maps which, almost surely in $\omega \in \Omega$, satisfy in the weak sense, the continuity equation 
\[
\begin{cases}
\partial_t m + \text{div}(\alpha m) = 0 \text{ in } (0,T)\times \T^d \\
m(0) = µ, m(T) = \nu_T.
\end{cases}
\]
together with the condition that they have to be adapted processes to the filtration $(\mathcal{F}_t)_{t \geq 0}$.\\

We now derive the associated HJB equation for different target processes $(\nu_t)_{t \geq 0}$.

\subsection{HJB equations of stochastic optimal transport}
 We define, formally for the moment, the value function $U$ by
\be\label{sotU}
U(t,µ,\nu) = \inf_{(\alpha,m)} \mathbb{E}\left[ \int_{t}^{T} \int_{\T^d}L(x,\alpha_s(x))m_s(dx)ds \bigg| \nu_{t} = \nu \right],
\ee
where the state process $(\alpha,m)$ has to satisfy the same requirement as previously except that the condition $m_0 = µ$ is now replaced by $m_t = µ$. Note that the expectation is conditioned on $\{\nu_{t} = \nu\}$. Another point of view consists in looking at $\nu$ as an uncontrolled state variable of the optimal control problem, that we try to attain at the final time with the controlled state variable.\\

Depending on the nature of the process $(\nu_s)_{s \geq 0}$, different HJB equations arise for the value $U$ in \eqref{sotU}. We now give a few examples of such equations.
\subsubsection{A constant target process}\hfill \\

 Observe briefly that in the simplest case in which the target process $(\nu_s)_{s \geq 0}$ is constant, we recover the usual optimal transport problem and the associated HJB equation is then \eqref{hjb1}.\\

\subsubsection{A Bernoulli like target process}\hfill \\

Consider a case in which at time $T /2$, the process $\nu$ is going to take the value $\nu_1$ with probability $p \in (0,1)$ and $\nu_2$ with probability $1-p$. It will then remain constant up to the final time $T$. In this context, after the time $T/2$, the problem is a standard (deterministic) optimal transport problem, whose value function we denote $U_{det}(t,µ,\nu)$. Because the problem is assumed to be risk neutral, we can compute the value of the problem at time $T /2$, just before the value of $\nu$ is revealed. It is simply given by
\be\label{condbernoulli}
U\left(\frac T 2,µ,\nu\right) = p U_{det}\left(\frac T2, µ, \nu_1\right) + (1-p)U_{det}\left(\frac T 2,µ,\nu_2\right),
\ee 
where $U_{det}$ is the value of the associated deterministic optimal transport problem. We can then compute the value $U$ for time $t \leq T/2$ by using the HJB equation
\[
- \partial_t U(t,\mu,\nu) + \mathcal{H}\left(\mu,D_µ U\right) = 0 \text{ in }  \left(0,\frac T 2\right) \times \mptd^2,
\]
together with the condition \eqref{condbernoulli}.
Let us remark that, in this setting, the value of the target process before $T/2$ does not matter.

 \subsubsection{A target process with jumps}\label{ex:jumps}\hfill \\
 
 Consider now a case in which the target process $(\nu_s)_{s \geq 0}$ jumps, at times $(s_n)_{n \geq 0}$ given by a Poisson process of intensity $\lambda : [0,T] \to \R_+$, from $\nu_{s_n}$ into $\mathcal{T}\nu_{s_n}$, where $\mathcal{T} : \mptd \to \mptd$ is a given operator. In such a situation, the associated HJB equation is given by
 \be\label{hjbl}
- \partial_t U(t,\mu,\nu) + \mathcal{H}\left(\mu,\nabla_µ U\right)  + \lambda(t) (U(t,µ,\nu) - U(t,µ,\mathcal{T}\nu))= 0 \text{ in } (0,T)\times \mptd^2.
 \ee

 Let us insist upon the fact that such type of target process can cover a wide range of models. For instance, if $\mathcal{T}$ is a constant operator, then the framework is quite close to the previous one and as at most two possible values of the target are possible, the initial one and its image by $\mathcal{T}$. Moreover, we could also consider cases involving more general jumps. For instance, assume that there is an independent sequence of times $(\tilde{s}_n)_{n \geq 0}$ given by a Poisson process of intensity $\lambda_2$, associated to an operator $\mathcal{T}_2$ such that the previous rule also applies but also the process jumps according to $\mathcal{T}_2$. In this case, the associated HJB equation would be
  \[
  \begin{aligned}
- \partial_t U(t,\mu,\nu) &+ \mathcal{H}\left(\mu,\nabla_µ U\right)  + \lambda(t) (U(t,µ,\nu) - U(t,µ,\mathcal{T}\nu))\\
 & + \lambda_2(t) (U(t,µ,\nu) - U(t,µ,\mathcal{T}_2\nu))= 0 \text{ in } (0,T)\times \mptd^2.
 \end{aligned}
 \]
 
 \subsubsection{The target process is pushed by a diffusion}\label{sec:mb}\hfill \\ 

 Consider now a case in which the target process $(\nu_s)_{s \geq 0}$ is given by $\nu_s = (\tau_{W_s})_{\#}\nu$ for a given $\nu \in \mptd$ and a process $(W_s)_{s \geq 0}$ (where $\tau_h(x) = x + h$ is the translation of $h$). We assume that $(W_s)_{s \geq 0}$ is given as the solution of the stochastic differential equation (SDE) 
 \be\label{sde}
 dW_t = \sigma(t)dB_t,
 \ee
 where $\sigma : \R_+ \to \R$ is a given function and $(B_t)_{t \geq 0}$ is a standard Brownian motion on $(\Omega, \mathcal{A},\mathbb{P},(\mathcal{F}_t)_{t \geq 0})$. In other words, the actual shape of the target measure is fixed by $\nu$, but it is constantly being translated by the process $(W_t)_{t \geq 0}$. Using an infinite dimensional analogue of It\^o's Lemma, such as in Cardaliaguet et al. \citep{cardaliaguet2019master} for instance, we deduce that the HJB equation characterizing the associated value $U$ is given by
 \be\label{hjbd2}
 \begin{aligned}
 -\partial_t U(t,\mu,\nu)& + \mathcal{H}\left(\mu,D_µ U\right) - \frac{\sigma^2(t)}{2}\int_{\T^{2d}}Tr\left[D^2_{\nu\nu}U\right]d\nu\otimes d\nu\\
 &-\frac{\sigma^2(t)}{2}\int_{\T^d}\text{div}_x(D_\nu U(t,µ,\nu,x))\nu(dx)  = 0 \text{ in } (0,T)\times \mptd^2. 
 \end{aligned}
 \ee
Contrary to the previous case, the present situation leads to terms involving derivatives of the value function with respect to the variable $\nu$ which represents the target measure. This is a general feature of such problems. In some particular situations, including this one as we shall see, the problem can be reduced in such a way that those terms do not appear, however in a general situation we cannot avoid to work directly with them.\\

 Let us remark that, formally, following the computations of Bertucci \citep{bertucci2021monotone2}, this HJB equation can be obtained as the limit of the case with jumps for well chosen operators $\mathcal{T}$ and jump rates $\lambda$.

 \subsubsection{The case of a stochastic cost functional}\hfill \\

 Consider now a slightly different setting. We now assume that the target process $(\nu_s)_{s \geq 0}$ is constant and that the randomness is carried in the cost function $L$. We assume that this randomness appears through a dependence on the value $w$ of a $d$ dimensional process $(W_s)_{s \geq 0}$ given as the solution of \eqref{sde}. In this situation, it is natural to consider a value function $U$ which also depends on the value $w$ of this process. To be more precise, we are considering the value $U$ defined by
 \be\label{sotUw}
U(t,µ,\nu,w) = \inf_{(\alpha,m)} \mathbb{E}\left[ \int_{t}^{T}\int_{\T^d}L(x,\alpha_s(x),W_s)m_s(dx)ds \bigg| W_{t} = w\right],
\ee
where the state process $(m_s)_{s \in [0,T]}$ evolves as in \eqref{sotU}, the infimum is carried over the same set.

 In this situation, the natural HJB equation satisfied by $U$ is
 \be\label{eqUnuw}
- \partial_t U(t,\mu,\nu,w) + \mathcal{H}\left(\mu,D_µ U,w\right) - \frac{\sigma^2(t)}{2}\Delta_w U= 0 \text{ in } (0,T)\times \mptd^2 \times \T^d.
 \ee

\subsubsection{Reduction of the case in which the target measure is pushed by a Brownian motion}\label{sec:reduw}
 
Equation \eqref{eqUnuw} leads us to the following simplification of the case in which the target measure is pushed by the process $(W_s)_{s \geq 0}$. Indeed, as we mentioned the shape of the final target is fixed at $\nu$ and thus only a finite dimensional parameter is sufficient to characterize it. More precisely, we want to make the formal change of variable
 \[
 U(t,µ,(\tau_{W_s})_{\#}\nu) = U(t,µ,W_s).
 \]
 This leads to the following HJB equation
  \be\label{hjbw}
 \begin{aligned}
- \partial_t U(t,\mu,w) + \mathcal{H}\left(\mu,D_µU\right) - \frac{\sigma^2(t)}{2}\Delta_w U = 0 \text{ in } (0,T)\times \mptd \times \T^d,
 \end{aligned}
 \ee
which is thus associated to the slightly more involved terminal condition
\[
U(T,µ,w) = \begin{cases} 0 \text{ if } µ = (\tau_w)_{\#}\nu,\\ + \infty \text{ else.}\end{cases}
\]

 \subsubsection{A comment on modelling}\hfill \\
 
Let us briefly comment on the choice we make to consider value functions as functions of both $µ$ and $ \nu$. There seemed to be a wide range of models for which keeping this distinction is not necessary: for instance if the cost function $L$ does not depend on the variable $x \in \T^d$. Indeed in this case, consider the equation \eqref{hjbl} and assume that $\mathcal{T}$ is a translation. Then studying \eqref{hjbl} is equivalent to studying
 \be\label{hjbl2}
 -\partial_t U(t,\mu) + \mathcal{H}\left(\mu,\nabla_µ U\right)  + \lambda(t) (U(t,µ) - U(t,\mathcal{T}^{-1}\mu))= 0 \text{ in } (0,T)\times \mptd.
 \ee
This could have also been observed on the case of Section \ref{sec:mb}, by considering the equation
 \[
 \begin{aligned}
- \partial_t U(t,\mu)& + \mathcal{H}\left(\mu,\nabla_µ U\right) - \frac{\sigma^2(t)}{2}\int_{\T^{2d}}Tr\left[D^2_{\mu\mu}U\right]dµ\otimes dµ\\
&-\frac{\sigma^2(t)}{2}\int_{\T^d}\text{div}_x(D_\mu U(t,µ,x))\mu(dx)  = 0 \text{ in } (0,T)\times \mptd. 
\end{aligned}
 \]
 
We believe that this type of simplification can be helpful in several cases. Although it seems that the intrinsic nature of the HJB equation associated to this stochastic optimal transport is by nature involving these two variables: a controlled one which yields a Hamiltonian, and an uncontrolled one which yields the term associated to the generator of the stochastic evolution of the target process.

 \section{A comparison principle for HJB equations on the set of probability measures}\label{sec:hjb}
 As we mentioned in the introduction, the aim of this paper is to study the stochastic optimal transport problem by means of the associated HJB equation. Two main mathematical difficulties arise in this approach. The first one consists in studying the HJB equation in itself while the second one lies in the characterization of the singular terminal condition. In this section, we focus on the first question and postpone the question of the terminal condition to the Section \ref{sec:sing}.
 
Here, we establish a general comparison principle for HJB equations on $\mptd$. We analyze first a pure HJB equation and we then explain how to extend it to HJB equations associated to stochastic optimal transport problems as the ones we mentioned before. This study is set on more general Hamiltonians than the one we have introduced before.

The notion of viscosity solution we introduce is different from the one usually used the literature. We believe that the present approach is better suited to study a wide range of problems. Furthermore, we justify this notion in the next section when defining rigorously the value functions.
\ca
 
 \subsection{Super-differentials of functions on $\mptd$}
Before presenting our notion of viscosity solution, we have to define a notion of super/sub-differential of functions on $\mptd$. Even though it is not the notion we are going to use, we start by recalling a common definition of super-differential.\\

In the literature, it is said that a function $\xi \in L^{1}((\T^d,µ),\R^d)$ belongs to the super-differential of $U : \mptd \to \R$ at $µ \in \mptd$ if for any $µ' \in \mptd$, $\pi \in \Pi^{opt}(µ,µ')$,\footnote{Recall that $\Pi^{opt}(µ,µ')$ is the set of optimal couplings between $µ$ and $µ'$ for the quadratic cost.}
\be\label{supclass}
U(µ') \leq U(µ) + \int_{\T^d\times \T^d}\xi(x)\cdot(y-x)\pi(dx,dy) + o(W_2(µ,µ')).
\ee
In such a situation we note $\xi \in \partial_{clas}^+U(µ)$. The sub-differential $\partial_{clas}^-U(µ)$ is defined as $\partial_{clas}^-U(µ)= -\partial_{clas}^+(-U)(µ)$. When $U$ is a smooth function, one recovers easily that $\partial_{clas}^+U(µ) = \partial_{clas}^-U(µ)= \{D_µU(µ)\}$.
\begin{Rem}
Let us insist on the fact that this notion of smoothness views $(\mptd,W_2)$ as a geometric space whose geodesics are the ones of the optimal transport with cost $c(x,y) = |x-y|^2$. Indeed, in \eqref{supclass}, we are considering optimal couplings between $µ$ and $µ'$. Looking at $\mptd$ as a flat space, would lead to consider super-differentials $\partial_{flat}^+U(µ)$ as the set of $\phi \in \mathcal{C}^0(\T^d,\R)$ such that for all $µ' \in \mptd$.
\be\label{eq:789}
U(µ') \leq U(µ) + \int_{\T^d}\phi(x)(µ'-µ)(dx) + o(W_2(µ,µ')).
\ee
In this case, we would have $\partial_{flat}^+U(µ) = \{\nabla_µ U(µ)\}$ for smooth functions $U$.
\end{Rem}
In this article we generalize the previous notion of super-differential in the spirit of Kantorovich's relaxation of \eqref{relaxedot}. The simplest way to proceed is to replace $\xi : \T^d \to \R^d$ by a map $\psi : \T^d \to \mprd$. Equivalently, we can consider a measure $\gamma \in \mathcal{P}(\T^d\times \R^d)$ whose first marginal is $(\pi_1)_{\#}\gamma = µ$ where $µ$ is the measure at which we are looking for an element of the super-differential. As in \eqref{supclass}, when considering the variations of a function $U$ between $µ'$ and $µ$ we have to consider a coupling between the two measures, and not only the difference as in \eqref{eq:789} for instance. We are not particularly interested in geodesics here so we shall not ask for the coupling to be optimal. Hence, we are lead to consider the following Definition.

\begin{Def}\label{def:sup}
Given an upper semi continuous function $U : \mptd \to \R$, we say that a measurable map $\psi : \T^d \to \mprd$ is in the super-differential of $U$ at the point $µ$ if
\begin{itemize}
\item There exists $C > 0$ such that, for all $x \in \T^d$, the support of $\psi(x,\cdot)$ is contained in $B(0_{\R^d},C)$.
\item For any $µ' \in \mptd$, for any $\gamma \in\Pi(µ,µ')$, the following holds
\be\label{eqdefsup}
U(µ') - U(µ) \leq \int_{\T^{2d}}\int_{\R^d}z\cdot(y-x)\psi(x,dz)\gamma(dx,dy) + o\left( \left(\int_{\T^{2d}}|x-y|^2\gamma(dx,dy)\right)^{\frac 12}  \right).
\ee
\end{itemize}
In this case we note $\psi \in \partial^+U(µ)$.
\end{Def}

\begin{Rem}
\begin{itemize}
\item The condition on the boundedness of the support of $\psi$ is too strong at the level of this definition and could have been replaced by $(x \to \int_{\R^d}z\psi(x,dz)) \in L^1(µ)$ so that \eqref{eqdefsup} still makes sense. However, since we are working on the bounded set $\T^d$, this condition will not be too restrictive for the rest of the analysis. Moreover, it will greatly help with the definition of the HJB equation on the elements of the super-differential. Hence we leave it here for convenience.
\item The inequality \eqref{eqdefsup} is a priori only carrying information when the term in the $o(\cdot)$ vanishes as $W_2(µ,µ') \to 0$.
\item If  $\psi \in \partial^+U(µ)$, then $x\to \int_{\R^d}z\psi(x,dz) \in \partial^+_{clas}U(µ)$.
\item If $\xi \in \partial^+_{clas}U(µ)$ then $x \to \delta_{\xi(x)} \in \partial^+U(µ)$.
\end{itemize}
\end{Rem}

For a lower semi continuous function $U$, we define $\partial^-U(µ) = \{ x \to (-Id)_{\#}\psi(x,\cdot) | \psi \in \partial^+U(µ)\}$.\\

We now provide what we believe to be an instructive computation around this notion of super-differentiability. This computation is based on Lions' Hilbertian approach. Consider a smooth function $\Phi: \mptd \to \R$, $X$ and $Y$ two $\T^d$ valued random variables on a standard probability space $(\Omega', A', \mathbb{P}')$, such that $\mathcal{L}(X) = µ$ and $\mathcal{L}(Y) = µ'$. We then want to evaluate the variations of $\Phi$ along the path $(m_s)_{s \in [0,1]}$ defined by $m_s = \mathcal{L}(X + s(Y - X))$.
\[
\begin{aligned}
\Phi(m_t) - \Phi(µ) &= \int_0^t\mathbb{E}_{\mathbb{P}'}[D_µ\Phi(m_s,X + s(Y-X)) \cdot (Y- X)]ds \\
&= \int_0^t\mathbb{E}_{\mathbb{P}'}[D_µ\Phi(µ,X ) \cdot (Y- X)]ds + o\left(\mathbb{E}_{\mathbb{P}'}[|Y-X|]\right),
\end{aligned}
\]
where $\mathbb{E}_{\mathbb{P}'}$ denotes the expectation on $(\Omega', A', \mathbb{P}')$. Because $\Phi$ is smooth, it defines a smooth mapping on the $\T^d$ valued random variables $\tilde{\Phi}:X \to \Phi(\mathcal{L}(X))$. On this last computation we see that the derivative $D_µ \Phi$ is linked to the gradient of $\tilde{\Phi}$. Because $\Phi$ is smooth, the gradient of $\tilde{\Phi}$ at $X$ is in fact of the form $\xi(X)$ for some map $\xi : \T^d \to \R^d$. The notion of super-differential we provided consists in looking for random variables in the super-differential of $\tilde{\Phi}$ without any particular restriction whereas in the super-differential $\partial^+_{clas}$, the random variables in the super-differential have to be a function of $X$.\\

The following Proposition states the super-differentiability of the squared Wasserstein distance. For the usual notion, such a result was already proved in Ambrosio and Gangbo \citep{ambrosio2008hamiltonian}, Proposition 4.3.
\begin{Prop}\label{prop:super}
For any $\mu,\nu \in \mptd$, $\gamma^o \in \Pi^{opt}(µ,\nu)$, consider the measurable map $\psi$ defined almost everywhere by the disintegration $(\pi_1,\pi_1 -\pi_2)_{\#}\gamma^o = µ(dx)\psi(x,dz)$. The function $ \Phi : µ' \to \frac{1}{2}W^2_2(µ',\nu)$ is such that $\psi \in \partial^+\Phi(µ)$.
\end{Prop}
\begin{proof}
Consider an arbitrary $\gamma \in \Pi(µ,µ')$ and its disintegration $\gamma(dx,dy) = µ(dx)k(x,dy)$. Recall that  $\gamma^o \in \Pi^{opt}(µ,\nu)$. Hence, $\int_{\T^d}\gamma^o(dx,dy) k(x,dz)$ is an admissible coupling between $µ'$ and $\nu$ (where the previous integral is taken with respect to $x$ only). Thus, by definition of $\Phi$
\[
\begin{aligned}
2\Phi(µ') -2 \Phi(µ) &\leq \int_{\T^{3d}} |y-z|^2\gamma^o(dx,dy)k(x,dz) - \int_{\T^{2d}} |x-z|^2\gamma^o(dx,dy)\\
&= \int_{\T^{3d}}|y-x + x -z|^2 - |x-z|^2\gamma^o(dx,dy)k(x,dz) \\
& = 2\int_{\T^{3d}} (y-x)\cdot (x-z) \gamma^o(dx,dy)k(x,dz) + \int_{\T^{2d}}|y-x|^2 \gamma(dx,dz).
\end{aligned}
\]
From which the result follows. Remark in particular that since $\T^d$ is bounded, the bound on the support is indeed verified.
\end{proof}
This result of everywhere super-differentiability justifies the use of the $2$-Wasserstein distance in the argument of doubling of variables that we are going to make afterwards to obtain a comparison principle.
\begin{Rem}
The previous result can be interpreted in the probabilistic or Hilbertian approach. It states that, given an optimal coupling $(X,Y)$ for the quadratic cost between $µ$ and $\nu$, we can consider an element of the super-differential of $\Phi$ which is constructed on the random variable $X-Y$ and not just on $X- \mathbb{E}_{\mathbb{P}'}[Y|X]$.
\end{Rem}

Another advantage of this definition of super-differential is that it makes more transparent the link with the so-called Hilbertian approach. Let $(\Omega', \mathcal{A}',\mathbb{P}')$ be an atomeless standard probabilistic space and $K$ be the set of $\T^d$ valued random variables from this space. We have the following result.
\begin{Prop}\label{prop:equiv}
Consider an usc function $U : \mptd \to \R$. Define $\tilde{U} : K \to \R$ by $\tilde{U}(X) = U(\mathcal{L}(X))$. Take $X \in K$ and assume that $Z \in \K^2(\Omega,\R^d)$ is such that for all $Y \in K$,
\[
\tilde{U}(Y) \leq \tilde{U}(X) + \mathbb{E}_{\mathbb{P}'}[Z\cdot (Y-X)] + o\left(\sqrt{\mathbb{E}_{\mathbb{P}'}[|X-Y|^2]}\right).
\]
Consider now a function $\psi : \T^d \to \mprd$ such that 
\[
\mathcal{L}(X,Z)(dx,dz) = \mathcal{L}(X)(dx)\psi(x,dz).
\]
Then, $\psi \in \partial^+U(\mathcal{L}(X))$.
\end{Prop}
\begin{proof}
Set $µ =\mathcal{L}(X)$. Take $\gamma \in \Pi(µ,µ')$ and desintegrate $\gamma(dx,dy)$ into $µ(dx)k(x,dy)$. Consider $\psi: \T^d\to \mprd$ such that $\mathcal{L}(X,Z)(dx,dz) = µ(dx)\psi(x,dz)$. Consider now $(X',Y',Z')$ such that $\mathcal{L}((X',Y',Z'))(dx,dy,dz)= µ(dx)\psi(x,dz)k(x,dy)$. Thanks to classical results, we can in fact consider a sequence $(X_n,Y_n,Z_n)_{n \geq 0}$ such that $\mathcal{L}((X_n,Y_n,Z_n))(dx,dy,dz)= µ(dx)\psi(x,dz)k(x,dy)$ and $\|(X_n,Z_n) -(X,Z)\|_{\infty} \leq n^{-1}$. It holds that
\be\label{eq:548}
U(µ') = \tilde{U}(Y_n) \leq U(µ) + \mathbb{E}_{\mathbb{P}'}[Z\cdot(Y_n-X)] + o\left(\sqrt{\mathbb{E}_{\mathbb{P}'}[|X-Y_n|^2]}\right).
\ee
On the other hand, 
\[
\begin{cases}
|\mathbb{E}_{\mathbb{P}'}[Z\cdot(Y_n-X)] - \mathbb{E}_{\mathbb{P}'}[Z_n\cdot(Y_n - X_n)] | \underset{n \to \infty}{\longrightarrow} 0,\\
\mathbb{E}_{\mathbb{P}'}[|X-Y_n|^2]  \underset{n \to \infty}{\longrightarrow} \int_{\T^{2d}}|x-y|^2\gamma(dx,dy). 
\end{cases}
\]
and 
\[
\mathbb{E}_{\mathbb{P}'}[Z_n\cdot(Y_n - X_n)] = \int_{\T^{2d}\times \R^d}z\cdot(y-x)\psi(x,dz)\gamma(dx,dy).
\]
Hence the result follows from passing to the limit $n \to \infty$ in \eqref{eq:548}.
\end{proof}
\begin{Rem}
In other words, we have elements of the super-differential of $U$ which describe all the elements of the super-differential of $\tilde{U}$.
\end{Rem}

\cred



 
 \ca

We end this section by explaining on a simple example how to consider super-differential of functions of more variables than a measure $µ \in \mptd$. For instance, consider the case of an additional time variable. Given $T > 0$ and an usc function $U : [0,T]\times \mptd \to \R$, we note $(\theta,\psi) \in \partial^+U(t,µ)$ if 
\[
\begin{aligned}
U(s,µ') \leq &U(t,µ) + \theta(s-t) + \int_{\T^{2d}}\int_{\R^d}z\cdot(y-x)\psi(x,dz)\gamma(dx,dy)+\\
& + o\left(|s-t|+\left( \int_{\T^{2d}}|x-y|^2\gamma(dx,dy)\right)^{\frac 12}  \right),
\end{aligned}
\]
for any $s \leq t, µ' \in \mptd, \gamma \in \Pi(µ,µ')$. Similarly we introduce $\partial^-U(t,µ) = \{(-\theta,x\to (-Id)_{\#}\psi(x,\cdot)) | (\theta,\psi) \in \partial^+(-U)(t,µ)\}$.

We are now able to define viscosity solutions of \eqref{hjb}.

 \subsection{Statement of the problem and definition of viscosity solutions}
 In this section we want to prove a comparison principle for HJB equations of the form
  \be\label{hjb}
 \partial_t U(t,\mu) + \mathcal{H}\left(t,\mu,D_µ U\right) = 0 \text{ in } (0,\infty)\times \mptd,
 \ee
where $\mathcal{H}: \R_+ \times \mptd \times (\T^d \to \R^d) \to \R$ is given by
\[
\mathcal{H}(t,µ,\xi) = \int_{\T^d}H(t,x,µ,\xi(x))µ(dx),
\]
where $H: \R_+ \times \T^d \times \mptd \times \R^d$ is given. In all this section, we shall assume that the following hypothesis holds.

\begin{hyp}\label{hyp:hgen}
The Hamiltonian $H$ satisfies
\begin{itemize}
\item $H$ is globally continuous.
\item There exists $C >0$ such that for all $p \in \R^d, s,t \in \R_+, µ,\nu \in \mptd, x,y \in \R^d$ we have
\[
|H(t,x,µ,p) - H(t,y,\nu,p) | \leq C(1+ |p|) (|t-s| + W_2(µ,\nu) + |x-y|).
\]
\end{itemize}
\end{hyp}
Let us insist upon the fact that, to simplify this section, we have reversed the sense of time compared to the HJB equations derived in Section \ref{sec:sot}.\\

As mentioned earlier, the expected candidate to be a solution of \eqref{hjb} is in general not smooth and we have to define a notion of weak solution. The natural techniques to study HJB equations such as \eqref{hjb} comes from the theory of viscosity solutions. The approach we provide here is somehow close to the one of Marigonda and Quincampoix \citep{marigonda2018mayer} in the sense that we provide an intrinsic proof of a comparison principle, and it is also close to the point of view of Lions' Hilbertian approach, also presented by Gangbo and Tudorascu \citep{gangbo2019differentiability}, in the sense that the notion of viscosity solution we are going to provide relies on ideas from this Hilbertian lifting. However, our result is more general than the ones of  Marigonda and Quincampoix \citep{marigonda2018mayer} and Gangbo and Tudorascu \citep{gangbo2019differentiability} because of the generality of equations we are able to treat. Moreover, we believe the proof we provide to be simpler.\\

In our definition of super-differential, an element of the super differential (with respect to the measure variable) is a map $\T^d \to \mprd$ and not simply a function $\T^d \to \R^d$. Thus, we have to specify how we want to evaluate $\mathcal{H}$ on such elements. We introduce here $\bar{\mathcal{H}} :\R_+ \times \mptd \times (\T^d \to \mprd)\to \R$ defined by
\[
\bar{\mathcal{H}}(t,µ,\psi) = \int_{\T^d\times \R^d}H(t,x,µ,y)µ(dx)\psi(x,dy),
\]
and we shall work with the following Definition.
\begin{Def}\label{def:visc}
An usc (resp. lsc) function $U : \R_+ \times \mptd \to \R$ is a viscosity sub (resp. super)-solution of \eqref{hjb} if, for any $t > 0, µ \in \mptd$ and $(\theta,\psi) \in \partial^+(U)(t,µ)$ (resp. $\partial^-U(t,µ)$) the following holds
\[
\theta + \bar{\mathcal{H}}(t,µ,\psi) \leq 0 \quad (\text{ resp. } \geq 0).
\]
 A viscosity solution of \eqref{hjb} is a locally bounded function such that $U^*$ is a viscosity sub-solution and $U_*$ is a viscosity super-solution.
\end{Def}
\begin{Rem}
The term $\bar{\mathcal{H}}(t,µ,\psi)$ is well defined because by definition of the super/sub-differential, $\psi(x,dz)$ has bounded support in $z\in\R^d$, uniformly in $x\in \T^d$.
\end{Rem}
\begin{Rem}
The choice we made to consider $\bar{\mathcal{H}}$ is not trivial. It shall be justified in the Section \ref{sec:cont} when proving that the value function of the stochastic optimal transport is indeed a viscosity solution of the HJB equation.
\end{Rem}

The main advantage of this notion of viscosity solutions, by comparison with other ones in the literature, is that it provides, relatively easily, a comparison principle, as we shall now see.

\subsection{Comparison principle and uniqueness of viscosity solutions}
As usual in the theory of viscosity solutions, uniqueness of solutions, $L^{\infty}$ estimates and other stability properties come from a comparison principle. We now establish such a result. 
\begin{Theorem}\label{thm:compgen}
Under Hypothesis \ref{hyp:hgen}, assume that $U$ and $V$ are respectively viscosity sub and super-solution of \eqref{hjb} such that for all $µ \in \mptd$, $U(0,µ) \leq V(0,µ)$. Then for all $t \geq 0, µ \in \mptd$, $U(t,µ) \leq V(t,µ)$.
\end{Theorem}
Following the standard ways to establish comparison principle, we are going to use the so-called technique of doubling of variables. We now present formally this technique. The proof of the Theorem is postponed to the end of this section. In this setting on $\mptd$, we introduce, for $\epsilon > 0$, the function
\[
(t,s,µ,µ') \to V(s,µ') - U(t,µ) + \frac{1}{2\epsilon}((t-s)^2 + W_2^2(µ,µ')),
\]
Arguing by contradiction, we will consider a point $(t^*,s^*,µ^*,\nu^*)$ of minimum of this function. Then, using Proposition \ref{prop:super}, we will prove that $\partial^+U(t^*,\mu^*)$ and $\partial^-V(s^*,\nu^*)$ are non empty. More precisely, we will be able to consider two elements, one in each of those sets, with some relation between them.\\

The next, final and main step of the proof consists in arriving at a contradiction by taking the difference of the viscosity relations, i.e. the relations given by the fact that $U$ is a sub-solution and $V$ a super-solution. Before presenting the proof of the final step, we prove the Lemma that we are going to use in order to consider elements of $\partial^+U(t^*,\mu^*)$ and $\partial^-V(s^*,\nu^*)$.
\begin{Lemma}\label{lemma:inc}
Consider an usc function $U$ and a continuous function $\Phi$ on $\R_+\times \mptd$ and $(t,µ) \in (0,\infty)\times \mptd$, a point of maximum of $U - \Phi$.\\ Then $(\theta,\psi)\in \partial^+\Phi(t,µ) \Rightarrow (\theta,\psi)\in \partial^+U(t,µ)$.
\end{Lemma}
\begin{proof}
Take $t>0,$ $µ\in \mptd$ such that $(t,µ) \in \text{argmax} \{U - \phi\}$ and also $(\theta,\psi) \in \partial^+\Phi(t,µ)$. For any $s\leq t, µ' \in \mptd$ and $\gamma \in \Pi(µ,µ')$,
\[
\ba
U(s,µ') - U(t,µ) &\leq \Phi(s,µ') - \Phi(t,µ)\\
& \leq \theta(s-t) + \int_{\T^{2d}}\int_{\R^d}z\cdot(y-x)\psi(x,dz)\gamma(dx,dy)\\
&\quad \quad + o\left(|s-t| + \left(\int_{\T^{2d}}|x-y|^2\gamma(dx,dy)\right)^{\frac 12}  \right)
\ea
\]
\end{proof}
We are now ready to prove the main result of this section.
\begin{proof}\textit{(Of Theorem \ref{thm:compgen}.)}
Assume that the result does not hold. Hence, there exists $T, \kappa > 0$ such that
\[
\inf_{t\leq T,µ\in \mptd} V(t,µ) - U(t,µ) \leq -\kappa.
\]
Thus, there exists $\rho> 0$ such that for any $\epsilon > 0$,
\be\label{eqpropeps}
\inf_{t,s\leq T,µ,µ'} V(s,µ') - U(t,µ) + \frac{1}{2\epsilon}\bigg((t-s)^2 + W_2^2(µ,µ')\bigg) + \rho(t+s) \leq -\frac{\kappa}{2}.
\ee
Since $U$ is usc and $V$ lsc, the previous infimum is reached at some point $(t_{\epsilon},s_{\epsilon},µ_{\epsilon},µ'_{\epsilon})$, because we are minimizing a lsc function on the compact $[0,T]^2\times \mptd^2$.\\

\textit{Step 1: using the viscosity properties.}
We treat first the case $t_{\epsilon},s_{\epsilon} > 0$. Take $\gamma_{\epsilon}^o \in \Pi^{opt}(µ_{\epsilon},µ'_{\epsilon})$ and denote $\psi_{\epsilon} : \T^d \to \mprd$ such that 
\[
(\pi_1,\epsilon^{-1}(\pi_1 - \pi_2))_{\#}\gamma_{\epsilon}^o(dx,dz) = µ_{\epsilon}(dx)\psi_{\epsilon}(x,dz).
\]
From Proposition \ref{prop:super} and Lemma \ref{lemma:inc}, we obtain that
\[
\begin{cases}
\left(\rho + \epsilon^{-1}(t_{\epsilon}-s_{\epsilon}), \psi_{\epsilon}\right) \in \partial^+U(t_{\epsilon},µ_{\epsilon}),\\
(-\rho - \epsilon^{-1}(s_{\epsilon}-t_{\epsilon}), \psi_{\epsilon}) \in \partial^-V(s_{\epsilon},µ'_{\epsilon}).
\end{cases}
\]
Since $U$ and $V$ are respectively sub and super-viscosity solution of \eqref{hjb}, we deduce that
\[
\ba
\rho + \epsilon^{-1}(t_{\epsilon} - s_{\epsilon})+ \int_{\T^{2d}}H\left(t_{\epsilon},x,µ_{\epsilon},\frac{x-y}{\epsilon}\right)\gamma^o_{\epsilon}(dx,dy) \leq 0,
\ea
\]
and that
\[
\ba
-\rho - \epsilon^{-1}(s_{\epsilon} - t_{\epsilon})+ \int_{\T^{2d}}H\left(s_{\epsilon},y,µ'_{\epsilon},\frac{x-y}{\epsilon}\right)\gamma^o_{\epsilon}(dx,dy) \geq 0.
\ea
\]

\textit{Step 2: Standard estimates.}\\
Taking the differences of the two previous inequalities leads to
\[
2 \rho \leq  \int_{\T^{2d} }H(s_{\epsilon},y,µ'_{\epsilon},\epsilon^{-1}(x-y))\gamma_\epsilon^o(dx,dy) - \int_{\T^{2d} }H(t_{\epsilon},x,µ_{\epsilon},\epsilon^{-1}(x-y))\gamma_{\epsilon}^o(dx,dy).
\]
Using the regularity assumptions we made on $H$, we deduce that
\[
\ba
2 \rho& \leq C\int_{\T^{2d} } (|t_{\epsilon}-s_{\epsilon}| + |x -y| + W_2(µ_{\epsilon} ,µ'_{\epsilon}))(1 + \eps^{-1}|x-y|)\gamma_{\epsilon}^o(dx,dy)\\
&\leq C\bigg( \eps^{-1}W_2^2(µ_{\epsilon},µ'_{\epsilon}) + (|t_{\epsilon} -s_{\epsilon}| + W_2(µ_{\epsilon},µ'_{\epsilon}))\eps^{-1}\int_{\T^d}|x-y|\gamma^o_{\epsilon}(dx,dy) \bigg).
\ea
\]
From standard estimates techniques of the method of doubling of variables, see for instance Crandall et al. \citep{crandall1992user}, we know that
\[
\frac{(t_{\epsilon} - s_{\epsilon})^2}{\epsilon} \longrightarrow 0, \frac{W_2^2(µ_{\epsilon},µ'_{\epsilon})}{\epsilon} \longrightarrow 0, \text{ as } \epsilon \to 0.
\]
Hence if the previous holds for all $\epsilon > 0$, we arrive at a contradiction by taking the limit $\epsilon \to 0$ since $\rho > 0$ was fixed independently of $\epsilon$.\\

\textit{Step 3: The minimum is at the boundary $t = 0$.}\\
From the previous Step, we deduce that, for $\epsilon$ small enough, the minimum is in fact reached at a point such that either $s_{\epsilon} = 0$ or $t_{\epsilon} = 0$. Since, we obtain from \eqref{eqpropeps} that $\lim_{\epsilon \to 0} \epsilon^{-1}(t_{\epsilon}-s_{\epsilon})^2 = \lim_{\epsilon \to 0} \epsilon^{-1}W^2_2(µ_{\epsilon},µ'_{\epsilon}) = 0$, we deduce that $\lim_{\epsilon \to 0}t_{\epsilon} = \lim_{\epsilon \to 0} s_{\epsilon} = 0$. Moreover, extracting a subsequence if necessary, there exists $µ_0$, limit of both $(µ_{\epsilon})_{\epsilon > 0}$ and $(µ'_{\epsilon})_{\epsilon > 0}$. Hence, using the lower semi continuity of $V$ and the upper semi continuity of $U$, we deduce that
\[
\begin{aligned}
V(0,µ_0) - U(0,µ_0) &\leq \liminf_{\epsilon\to 0} V(t_{\epsilon},µ_{\epsilon}) - U(s_{\epsilon},µ'_{\epsilon})  - \frac{1}{2\epsilon}((t_{\epsilon}-s_{\epsilon})^2 + W_2^2(µ_{\epsilon},µ'_{\epsilon})) - \alpha(t_{\epsilon}+s_{\epsilon})\\
&\leq - \frac{\kappa}{2}.
\end{aligned}
\]
The previous being clearly a contradiction, we finally deduce that the Theorem is true.
\end{proof}
 \begin{Rem}
 In the previous, we omitted to treat the particular case in which the minimum is reached for either $t_\eps$ or $s_\eps$ equal to $T$. This raises no difficulty. It can be treated by either adding a term of the form $\eps'(T-t)^{-1}$ for some $\eps' \in \R$ and then letting $\eps' \to 0$ or either by simply changing the notion of super (or sub)-differential in the $t$ variable, so that only perturbations with smaller time are taken into account.
 \end{Rem}

From the comparison principle easily follows the following result of uniqueness.
\begin{Theorem}
Under Hypothesis \ref{hyp:hgen}, given a continuous initial condition $U_0$, there exists at most one viscosity solution $U$ of \eqref{hjb} such that for all $µ \in \mptd$, $U^*(0,µ) = U_0(µ) = U_*(0,µ)$.
\end{Theorem}
\begin{proof}
By considering two such solutions $U$ and $V$, using the comparison principle, we immediately arrive at the fact that $V \leq U \leq V$ which proves the claim.
\end{proof}
More generally, we can use the comparison principle to establish stability results or $L^{\infty}$ estimates. For instance, the following is an immediate corollary of Theorem \ref{thm:compgen}.
\begin{Cor}
Under Hypothesis \ref{hyp:hgen}, consider a viscosity sub-solution $U$ and a viscosity super-solution $V$ of \eqref{hjb}, then $t \to \max_µ\{U(t,µ) - V(t,µ)\}$ is non increasing.
\end{Cor}

\subsection{Extensions to other HJB equations}
We now explain how to make use of the previous results, or more precisely of their proofs, to study the equations \eqref{hjbl} and \eqref{hjbw}.

\subsubsection{The case of jumps}
We focus here on \eqref{hjbl}. Formally, it suffices to remark that the terms in $\lambda$ in \eqref{hjbl} do not involve derivatives of the solution and thus are quite easy to treat. Moreover, the fact that the functions depend here on two measures instead of one does not perturb the previous argument as we shall now see. The definition of viscosity solutions of \eqref{hjbl} takes the following form 

\begin{Def}\label{def:viscl}
An usc (resp. lsc) function $U: [0,T] \times \mptd^2\to \R$ is said to be a viscosity sub-solution (resp. super-solution) of \eqref{hjbl} if for any $t\in[0,T), µ,\nu \in \mptd$ and $(\theta,\psi,\psi') \in \partial^+U(t,µ,\nu)$ (resp. $\in \partial^-U(t,µ,\nu)$)
\[
 -\theta +  \bar{\mathcal{H}}(t,µ,\psi)+ \lambda(t)(U(t,µ,\nu) - U(t,µ,\mathcal{T}\nu)) \leq 0 \text{ (resp. } \geq 0 ). 
\]
A viscosity solution of \eqref{hjbl} is a locally bounded function $U$ such that $U_*$ is a viscosity super-solution and $U^*$ is a viscosity sub-solution.
\end{Def}
\begin{Rem}
Of course the existence of an element in the super-differential in the $\nu$ variable is useless here, and could be removed.
\end{Rem}
As in the previous case, a comparison principle can be stated.
 \begin{Prop}\label{prop:compl}
Under Hypothesis \ref{hyp:hgen}, assume that $U$ and $V$ are respectively viscosity sub and super-solutions of \eqref{hjbl} such that for all $µ,\nu$, $U(T,µ,\nu) \leq V(T,µ,\nu)$ and such that they are both bounded functions. Assume also that $\lambda$ is a continuous non negative function. Then for all time $t \in [0,T]$ and measures $µ,\nu \in \mptd$, $U(t,µ,\nu) \leq V(t,µ,\nu)$.
\end{Prop}
\begin{proof}
We only explain how the addition of the term in $\lambda$ perturbs the proof of Theorem \ref{thm:compgen}.
As in the previous proof, we consider the function
\[
Z(t,s,µ,\nu,µ',\nu') := V(t,µ,\nu) - U(s,µ',\nu') + \frac{1}{2\epsilon} ( W_2^2(µ,µ') + W_2^2(\nu,\nu') + (t-s)^2) + \rho(2T -t-s).
\]
Considering a point of minimum $(t_{\epsilon},s_{\epsilon},µ_{\epsilon},\nu_{\epsilon},µ_{\epsilon}',\nu_{\epsilon}')$ of $Z$, if $t_{\epsilon},s_{\epsilon} < T$, and arguing exactly as we did before we easily arrive at
\[
\begin{aligned}
2 \rho + \lambda(s_{\epsilon})(U(s_{\epsilon},µ'_{\epsilon},\nu'_{\epsilon})-U(s_{\epsilon},µ'_{\epsilon},\mathcal{T}\nu'_{\epsilon})) - \lambda(t_{\epsilon})(V(t_{\epsilon},µ_{\epsilon},\nu_{\epsilon}) - V(t_{\epsilon},µ_{\epsilon},\mathcal{T}\nu_{\epsilon})) \leq o(1),
\end{aligned}
\]
where the right side term vanishes as $\epsilon \to 0$. Let us compute
\[
\begin{aligned}
\lambda&(s_{\epsilon})(U(s_{\epsilon},µ'_{\epsilon},\nu'_{\epsilon})-U(s_{\epsilon},µ'_{\epsilon},\mathcal{T}\nu'_{\epsilon})) - \lambda(t_{\epsilon})(V(t_{\epsilon},µ_{\epsilon},\nu_{\epsilon}) - V(t_{\epsilon},µ_{\epsilon},\mathcal{T}\nu_{\epsilon})) \\
&\geq -C |\lambda(t_{\epsilon}) - \lambda(s_{\epsilon})| + \lambda(s_{\epsilon})(U(s_{\epsilon},µ'_{\epsilon},\nu'_{\epsilon})-U(s_{\epsilon},µ'_{\epsilon},\mathcal{T}\nu'_{\epsilon}) -(V(t_{\epsilon},µ_{\epsilon},\nu_{\epsilon}) - V(t_{\epsilon},µ_{\epsilon},\mathcal{T}\nu_{\epsilon})))\\
&\geq -C |\lambda(t_{\epsilon}) - \lambda(s_{\epsilon})| + \lambda(s_{\epsilon})\frac{1}{2\epsilon}(W_2^2(\mathcal{T}\nu_{\epsilon},\mathcal{T}\nu'_{\epsilon}) - W_2^2(\nu_{\epsilon},\nu'_{\epsilon}))\\
&\geq - C |\lambda(t_{\epsilon}) - \lambda(s_{\epsilon})| - \lambda(s_{\epsilon})\frac{1}{2\epsilon}W_2^2(\nu_{\epsilon},\nu'_{\epsilon}).
\end{aligned}
\]
Remark that in the previous, $C$ only depends on the bounds on $U$ and $V$. The limit of the last lower bound in the previous chain of inequalities is $0$ as $\epsilon \to 0$. Indeed, as in the previous proof of the comparison principle, we also recover that $\lim_{\epsilon \to 0} s_{\epsilon}-t_{\epsilon} =\lim_{\epsilon \to 0} \epsilon^{-1}W_2^2(\nu_{\epsilon},\nu'_{\epsilon})= 0$. Hence, using the continuity of $\lambda$ we obtain the required result by following the same argument as in the proof of Theorem \ref{thm:compgen}.
\end{proof}
\begin{Rem}
No assumption on $\mathcal{T}$ is needed here.
\end{Rem}

\subsubsection{The case of a target measure being pushed by a diffusion}
We now turn to the case of \eqref{hjbw}. This equation being of second order in $w$, the definition of viscosity solution is more involved. Indeed, because we are interested in viscosity solutions of a second order HJB equation, we need to introduce super-jets. We only consider particular forms of super-jets, namely only ones which are of interest for our problem, which is only of second order in the $w$ variable. 

For a function $U : [0,T]\times \mptd \times \T^d \to \R$, and $(t,µ,w) \in [0,T)\times \mptd \times \T^d$, the super-jet $J^+(U)(t,µ,w)$ of $U$ at $(t,µ,w)$ is defined as the set of $(\theta,\psi,p,X) \in \R\times (\T^d \to \mprd)\times\R^d\times S_d(\R)$ such that for any $t'\geq t,µ' \in \mptd, w' \in \T^d$ and $\gamma \in \Pi(µ,µ')$,
\[
\ba
U(t',µ',w') \leq &U(t,µ,w)+ \theta(t' - t)+ \int_{\T^{2d}}\int_{ \R^d}z\cdot(y-x)\gamma(dx,dy)\psi(x,dz)\\
& + p\cdot (w'-w) + (w'-w)\cdot X\cdot (w'-w)\\
& + o\left(|t'-t|+ \left(\int_{\T^{2d}}|x-y|^2d\gamma(dx,dy)\right)^{\frac 12} +  |w'-w|^2\right).
\ea
\]
Note that this notion of super-jet might not seem to be the most natural at first glance, since we omitted  the cross derivatives terms involving $w$ and $t$ or $µ$, which resumes to considering only super-jet in which those elements vanish. However, this notion is sufficient for the analysis we provide here, in particular because of the doubling of variables we are going to take.

As we did for superdifferentials, we define $$J^-(U)(t,µ,w) = \{(\theta,\gamma,p,X), (-\theta, (x \to (-Id)_{\#}\psi(x,\cdot)),-p,-X) \in J^+(-U)(t,µ,w) \}.$$

We can now introduce the notion of viscosity solutions of \eqref{hjbw}.
\begin{Def}\label{def:viscw}
An usc (resp. lsc) function $U$ is a viscosity sub-solution (resp. super-solution) of \eqref{hjbw} if for any $(t,µ,w) \in [0,T)\times \mptd\times \T^d$, $(\theta,\psi,p,X) \in J^+(U)(t,µ,w)$ (resp. $J^-(U)(t,µ,w)$) 
\[
\ba
-\theta&+  \bar{\mathcal{H}}\left(t,\mu,\psi\right) -\frac{\sigma^2(t)}{2}Tr(X)\leq 0\text{ (resp. } \geq 0 ).
\ea
\]
A viscosity solution of \eqref{hjbw} is a bounded function $U$ such that $U_*$ is a super-solution and $U^*$ is a sub-solution.
\end{Def}

Once again, a comparison principle result holds for this type of equation.
\begin{Prop}\label{prop:compw}
Assume that in addition to Hypothesis \ref{hyp:hgen}, there exists $C > 0$ such that for all $t \in [0,T],x \in \T^d, µ \in \mptd, p \in \R^d$
\be\label{hyphsigma}
\ba
|H(t,x,µ,p)| \leq C(1 + |p|^2),\\
|D_pH(t,x,µ,p)| \leq C(1 + |p|).
\ea
\ee
Let $U$ and $V$ be respectively a bounded viscosity sub-solution and a bounded viscosity super-solution of \eqref{hjbw}. If $U(0,µ,w) \leq V(0,µ,w)$, then $U \leq V$.
\end{Prop}
\begin{proof}
As usual in viscosity solution theory, we argue by contradiction and we assume that there exists $\kappa, \rho > 0$, such that for any $\epsilon > 0$
\[
\inf \left\{ V(s,µ',w') - U(t,µ,w) + \frac{1}{2\epsilon}\bigg((t-s)^2 + W_2^2(µ,µ') + |w-w'|^2\bigg) + \rho (2T -t-s)\right\} \leq -\kappa,
\]
where the infimum is taken over all $s,t \leq T, w,w' \in \T^d, µ,µ' \in \mptd$. \\

\textit{Step 1: Reformulation of the problem in the Hilbert space.}\\
As \eqref{hjbw} involves second order terms, we a priori need to use similar techniques as in \citep{crandall1992user} to conclude. Hence, we build on Lions' Hilbertian approach to transform the problem.
 
Let us consider an atomeless probabilistic space $(\Omega', \mathcal{A}', \mathbb{P}')$ and define $\tilde{V}(t,X,w) = V(t,\mathcal{L}_{\T^d}(X),w)$ and $\tilde{U}(t,X,w) = U(t,\mathcal{L}_{\T^d}(X),w)$ for $X \in \K^2(\Omega',\R^d)$, and where, for the rest of this proof, for $X \in \K^2(\Omega',\R^d)$, $\mathcal{L}_{\T^d}(X) = \mathcal{L}(p_{\T^d}(X))$ where $p_{\T^d} : \R^d \to \T^d$ is the natural projection. Recall that $\K^2(\Omega',\R^d)$ is a separable Hilbert space. \\

We can now consider
\[
\ba
\Phi(t,s,X,Y,w,w') := \tilde{V}(s,Y,w')& - \tilde{U}(t,X,w) + \frac{1}{2\epsilon}\bigg((t-s)^2 + \mathbb{E}_{\mathbb{P}'}[|X-Y|^2] + |w-w'|^2\bigg)+\\
 &+ \rho (2T -t-s) + \alpha(\eta(X) + \eta(Y)),
\ea
\]
where $\eta(X) := \sqrt{1 + \mathbb{E}_{\mathbb{P}'}[|X|^2]}$. We obtain that, for $\alpha> 0$ sufficiently small,
\[
\inf \bigg\{\Phi(t,s,X,Y,w,w') | t,s,\leq T, X,Y \in \K^2(\Omega',\R^d),w,w' \in \T^d\bigg\}\leq -\frac{\kappa}2.
\]
Thanks to Stegall's Lemma \citep{stegall1978optimization}, we know that for any $\delta > 0$, there exist $\beta_1,\beta_2 \in \R, Z_1,Z_2 \in \K^2(\Omega',\R^d)$, $|\beta_1|, |\beta_2|, \mathbb{E}_{\mathbb{P}'}[|Z_1|^2], \mathbb{E}_{\mathbb{P}'}[|Z_2|^2]\leq \delta$ and 
\[
\Phi(t,s,X,Y,w,w') + \mathbb{E}_{\mathbb{P}'}[Z_1\cdot X + Z_2 \cdot Y] + \beta_1 t + \beta_2 s
\]
has a unique strict minimum at point $(\bar{t},\bar{s},\bar{X},\bar{Y},\bar{w},\bar{w}')$. The case in which the minimum is reached for $\bar{t} = T$ or $\bar{s}= T$ can be treated as in the proof of Theorem \ref{thm:compgen} to arrive at a similar contradiction and we do not reproduce it.

Hence we assume that $\bar{t},\bar{s} < T$. Our goal is to use Lemma 4 in Lions \citep{lions1989viscosity}, to consider appropriate elements in the super-jets. Note that the function we consider is not defined on an Hilbert space a priori but since Lemma 4 in \citep{lions1989viscosity} is only a local result, we can consider that it is the case since $[0,T]\times \T^d$ is locally similar to $ \R^{d+1}$. We consider now an orthonormal basis of $\R\times \K^2(\Omega',\R^d)\times \R^d$ such that the first $d$ elements are given by $(0,0,e_i)$ where the $(e_i)$ are the elements of the canonical basis of $\R^d$. Using finally Lemma 4 in \citep{lions1989viscosity}, we obtain that there exist 
\begin{itemize}
\item Two matrices $S,S' \in S_d(\R)$ such that $S \leq S'$.
\item A sequence $(t_n,s_n,X_n,Y_n,w_n,w'_n)_{n\geq 0}$ converging toward $(\bar{t},\bar{s},\bar{X},\bar{Y},\bar{w},\bar{w}')$.
\item A sequence $(\omega_n, \omega'_n,\xi_n,\xi'_n,p_n,p'_n,A_n, B_n)_{n \geq 0}$ converging toward $0$.
\end{itemize}
such that 
 \footnote{We use the equivalent notation for the standard super jets of the functions $\tilde{U}$ and $\tilde{V}$ to avoid introducing a new one.}
\be\label{eq:882}
\ba
&(\beta_1-\rho +\epsilon^{-1}(\bar{t}-\bar{s}) + \omega_n, \epsilon^{-1}(\bar{X}-\bar{Y}) + \xi_n + Z_1 + \alpha\nabla\eta(\bar{X}), \epsilon^{-1}(\bar w - \bar w ')+p_n,S + A_n)\\
&\quad  \in J^+(\tilde{U})(t_n,X_n,w_n),\\
&(-\beta_2 + \rho + \epsilon^{-1}(\bar{t}-\bar{s}) + \omega'_n, \epsilon^{-1}(\bar{X}-\bar{Y}) + \xi'_n-Z_2 - \alpha\nabla\eta(\bar{Y}), \epsilon^{-1}(\bar w - \bar w ')+p'_n,S' + B_n)\\
& \quad \in J^-(\tilde{V})(s_n,Y_n,w'_n).
\ea
\ee

\textit{Step 2: Coming back to the original formulation.}\\
Thanks to Proposition \ref{prop:equiv}, \eqref{eq:882} implies in particular that
\[
\ba
&(\beta_1-\rho +\epsilon^{-1}(\bar{t}-\bar{s}) + \omega_n, \psi_n, \epsilon^{-1}(\bar w - \bar w') +p_n,S + A_n) \in J^+(U)(t_n,\mathcal{L}_{\T^d}(X_n),w_n),\\
&(-\beta_2 + \rho + \epsilon^{-1}(\bar{t}-\bar{s})+ \omega'_n, \varphi_n, \epsilon^{-1}(\bar w - \bar w') + p'_n,S' + B_n) \in J^-(V)(s_n,\mathcal{L}_{\T^d}(Y_n),w'_n),
\ea
\]
where $\psi_n$ and $\varphi_n$ satisfy
\[
\ba
\mathcal{L}(p_{\T^d}(X_n),\epsilon^{-1}(\bar{X}-\bar{Y}) + \xi_n+ Z_1 +\alpha \nabla\eta(\bar{X}))(dx,dz) = \mathcal{L}_{\T^d}(X_n)(dx)\psi_n(x,dz),\\
\mathcal{L}(p_{\T^d}(Y_n),\epsilon^{-1}(\bar{X}-\bar{Y}) + \xi'_n-Z_2 -\alpha \nabla\eta(\bar{Y}))(dx,dz) = \mathcal{L}_{\T^d}(Y_n)(dy)\varphi_n(x,dz).
\ea
\]

Using the fact that $U$ is a subsolution of \eqref{hjbw}, we obtain that
\be\label{viscU}
\ba
\rho -\epsilon^{-1}(\bar{t}-\bar{s}) -\beta_1+ \omega_n + \bar{\mathcal{H}}(t_n,\mathcal{L}_{\T^d}(X_n),\psi_n)-\frac{1}{2}Tr(S + A_n) \leq0.
\ea
\ee
Using the fact that $V$ is a super solution, we obtain that
\be\label{viscV}
-\rho - \epsilon^{-1}(\bar{t}-\bar{s})+ \beta_2 + \omega'_n + \bar{\mathcal{H}}(s_n,\mathcal{L}_{\T^d}(Y_n),\varphi_n)-\frac{1}{2}Tr(S' + B_n) \geq0.
\ee
Let us compute
\[
\ba
&|\bar{\mathcal{H}}(t_n,\mathcal{L}_{\T^d}(X_n),\psi_n) - \mathbb{E}_{\mathbb{P}'}[H(\bar{t},p_{\T^d}(\bar{X}),\mathcal{L}_{\T^d}(\bar{X}),\epsilon^{-1}(\bar{X}-\bar{Y})+ Z_1 + \alpha\nabla\eta(\bar{X}) )]|\\
&=|\mathbb{E}_{\mathbb{P}'}[H(t_n,p_{\T^d}(X_n),\mathcal{L}_{\T^d}(X_n),\epsilon^{-1}(\bar{X}-\bar{Y}) + \xi_n+ Z_1 +\alpha \nabla\eta(\bar{X}))]\\
&\quad \quad \quad \quad - \mathbb{E}_{\mathbb{P}'}[H(\bar{t},p_{\T^d}(\bar{X}),\mathcal{L}_{\T^d}(\bar{X}),\epsilon^{-1}(\bar{X}-\bar{Y})+ Z_1 + \alpha\nabla\eta(\bar{X}) )] |.
\ea
\]
From the growth assumption on $H$ and the dominated convergence Theorem, we deduce that the previous difference vanishes as $n \to \infty$. Hence, we can pass to the limit $n \to \infty$ in \eqref{viscU} and \eqref{viscV} and we obtain
\[
\ba
\rho -\epsilon^{-1}(\bar{t}-\bar{s})  - \beta_1+\mathbb{E}_{\mathbb{P}'}[H(\bar{t},p_{\T^d}(\bar{X}),\mathcal{L}_{\T^d}(\bar{X}),\epsilon^{-1}(\bar{X}-\bar{Y})+ Z_1 + \alpha\nabla\eta(\bar{X}) )]-\frac{1}{2}Tr S  \leq0,\\
-\rho - \epsilon^{-1}(\bar{t}-\bar{s}) + \beta_2 + \mathbb{E}_{\mathbb{P}'}[H(\bar{s},p_{\T^d}(\bar{Y}),\mathcal{L}_{\T^d}(\bar{Y}),\epsilon^{-1}(\bar{X}-\bar{Y})- Z_2 - \alpha\nabla\eta(\bar{Y}) )]-\frac{1}{2}Tr S'  \geq0.
\ea
\]
\textit{Step 3: Standard viscosity solutions estimates.}\\
Taking the difference of the two previous inequalities yields
\[
\ba
2 \rho &\leq \beta_1 + \beta_2+ \mathbb{E}_{\mathbb{P}'}[H(\bar{t},p_{\T^d}(\bar{X}),\mathcal{L}_{\T^d}(\bar{X}),\epsilon^{-1}(\bar{X}-\bar{Y})+ Z_1 + \alpha\nabla\eta(\bar{X}) )]\\
& \quad - \mathbb{E}_{\mathbb{P}'}[H(\bar{s},p_{\T^d}(\bar{Y}),\mathcal{L}_{\T^d}(\bar{Y}),\epsilon^{-1}(\bar{X}-\bar{Y})- Z_2 - \alpha\nabla\eta(\bar{Y}) )   ].
\ea
\]
Recalling Hypothesis \ref{hyp:hgen} and \eqref{hyphsigma}, we can estimate
\[
\ba
H&(\bar{t},p_{\T^d}(\bar{X}),\mathcal{L}_{\T^d}(\bar{X}),\epsilon^{-1}(\bar{X}-\bar{Y})+ Z_1 + \alpha\nabla\eta(\bar{X}) )   - H(\bar{s},p_{\T^d}(\bar{Y}),\mathcal{L}_{\T^d}(\bar{Y}),\epsilon^{-1}(\bar{X}-\bar{Y})- Z_2 - \alpha\nabla\eta(\bar{Y}) ) \\
=&  H(\bar{t},p_{\T^d}(\bar{X}),\mathcal{L}_{\T^d}(\bar{X}),\epsilon^{-1}(\bar{X}-\bar{Y})+ Z_1 + \alpha\nabla\eta(\bar{X}) )  - H(\bar{s},p_{\T^d}(\bar{Y}),\mathcal{L}_{\T^d}(\bar{Y}),\epsilon^{-1}(\bar{X}-\bar{Y})+ Z_1 + \alpha\nabla\eta(\bar{X}) )\\
& +H(\bar{s},p_{\T^d}(\bar{Y}),\mathcal{L}_{\T^d}(\bar{Y}),\epsilon^{-1}(\bar{X}-\bar{Y})+ Z_1 + \alpha\nabla\eta(\bar{X}) )- H(\bar{s},p_{\T^d}(\bar{Y}),\mathcal{L}_{\T^d}(\bar{Y}),\epsilon^{-1}(\bar{X}-\bar{Y})- Z_2 - \alpha\nabla\eta(\bar{Y}) )\\
 \leq & C\bigg(1 + |\epsilon^{-1}(\bar{X}-\bar{Y})+ Z_1 + \alpha\nabla\eta(\bar{X})|\bigg)\bigg(|\bar t - \bar s| + W_2(\mathcal{L}(\bar{X}), \mathcal{L}(\bar Y)) + |\bar X- \bar Y|\bigg)\\
& + C \bigg(1 + |\epsilon^{-1}(\bar{X}-\bar{Y})|\bigg)(| Z_1 + \alpha\nabla\eta(\bar{X}) + Z_2 + \alpha \nabla \eta(\bar Y)|\bigg),
\ea\]
where the last line is obtained by assuming that $\alpha, \delta \leq 1$, which we can do without loss of generality.
We then deduce that
\[
\ba
2 \rho \leq&\beta_1 + \beta_2 +C \mathbb{E}_{\mathbb{P}'}[( 1+ \epsilon^{-1}|\bar{X}-\bar{Y}|)(|\bar{X}-\bar{Y}|+ |\bar{t} - \bar{s}| + \sqrt{\mathbb{E}_{\mathbb{P}'}[|\bar X-\bar Y|^2]})]\\
& + C\mathbb{E}_{\mathbb{P}'}[(1 + \epsilon^{-1}|\bar{X}-\bar{Y}|)(|Z_1 + \alpha \nabla\eta(\bar{X}) +Z_2 + \alpha \nabla\eta(\bar{Y})| )].
\ea
\]
From the same argument as in the proof of Theorem \ref{thm:compgen}, we obtain that the third term of the right side vanishes as $\epsilon \to 0$, uniformly in $\alpha$ and $\delta$, recall that $\delta$ measures the size of $\beta_1,\beta_2,Z_1$ and $Z_2$. From Cauchy Schwarz inequality we finally deduce that the second term of the right side is bounded by 
\[
C\left(1 + \epsilon^{-1}\sqrt{\mathbb{E}_{\mathbb{P}'}[|\bar{X} - \bar{Y}|^2]}\right)(\delta + \alpha).
\]
Hence, taking the limits $\delta \to 0$ then $ \alpha \to 0$ then $\epsilon \to 0$, we conclude that $2\rho \leq 0$, which is a contradiction.\\

 In definitive we have indeed proven that $U \leq V$.
\end{proof}
\begin{Rem}
The growth assumptions on $H$ specified in the statement of Proposition \ref{prop:compw} seems to be removable by the use of techniques which are not particularly new to viscosity solutions theory. However, since such questions are not the core ones of our paper, we leave them for future research.
\end{Rem}

These comparison principles are essential tools to characterize functions as viscosity solutions of equations of the form of \eqref{hjb1}. Would the terminal conditions in our problems be continuous functions, the previous results would be enough to develop a a proper theory of \eqref{hjb1}. However because of the singularity that we expect at the origin, we shall have to characterize the behaviour of the solution near $t = T$ to have proper comparison principle. Namely, we shall use the following result.
\begin{Theorem}\label{thm:compinit}
Under the assumptions of Proposition \ref{prop:compl}, consider $U$ and $V$, two viscosity solutions of \eqref{hjbl}, locally bounded in $[0,T)\times \mptd^2$ such that 
\[
\lim_{t \to T} \sup_{µ,\nu \in \mptd} |U(t,µ,\nu) - V(t,µ,\nu)| = 0.
\]
Then $U= V$.
\end{Theorem}
\begin{proof}
By a symmetry argument, it is sufficient to prove $U \leq V$. Assume that it is not the case, hence that there exists $\kappa > 0$ and $t \in [0,T), µ,\nu \in \mptd$ such that
\[
U(t,µ,\nu) - V(t,µ,\nu) > \kappa.
\]
Take $\delta > 0$. By exactly the same argument as in the proof of Proposition \ref{prop:compl}, we deduce that there exists $\rho > 0$, such that for $\epsilon > 0$ small enough the minimum of the function 
\[
(t,s,µ,\nu,µ',\nu') \to V(t,µ,\nu) - U(s,µ',\nu') + \frac{1}{2\epsilon}((t-s)^2 + W_2^2(µ,µ') + W_2^2(\nu,\nu')) + \rho(2T-s-t)
\]
on the set $[0,T-\delta]^2\times \mptd^4$, is reached for either $t = T-\delta$ or $s = T-\delta$ (recall that $U$ and $V$ are both bounded on $[0,T-\delta]\times \mptd^2$). Taking the limit $\epsilon \to 0$, we deduce that
\[
\sup_{µ,\nu \in \mptd} U(T-\delta,µ,\nu) - V(T-\delta,µ,\nu) > \kappa.
\]
Taking the limit $\delta \to 0$, we obtain a contradiction and thus the result is proved.
\end{proof}
The same type of result obviously holds true for \eqref{hjbw}.

\subsection{Comments on our notion of viscosity solution}
In recent years, the study of HJB equations on the set of probability measure been the subject to a lot of works which have failed to establish general comparison principles for HJB equations associated to stochastic problems. On the other hand, the study of HJB equations set on an Hilbert space is a problem which is for the most part solved at the moment, except of course for new singular problems.

We believe that our approach provides a link between the two problems, namely through the notion of super-differential which we have chosen. In our opinion, this is a strong justification of the so-called Hilbertian approach developed by P.-L. Lions, originally to study MFG master equation. 

Recall that, in this approach, a probabilistic space $(\Omega', \mathcal{A}',\mathbb{P}')$ is fixed, and the study of \eqref{hjb} is replaced by the study of
\[
\partial_t \tilde{U} + \mathbb{E}_{\mathbb{P}'}[H(t,X,\mathcal{L}(X),\nabla \tilde{U})] = 0 \text{ in } (0,\infty)\times \K^2(\Omega',\R^d),
\]
where formally we have made the change of variable $\tilde{U}(X) = U(\mathcal{L}_{\T^d}(X))$.

This approach hints strongly the notion of super differential that we took. But maybe more importantly, it provides an interpretation for this HJB equation in the Hilbert space. Indeed, the Hilbertian approach can be interpreted as a process of labelling all the elements of mass of the measures, namely by labels $\omega \in \Omega$. This procedure allows to split mass, by assigning to the elements $\omega$ and $\omega'$ different velocities $Z(\omega)$ and $Z(\omega')$ even if they are in the same location, that is $X(\omega) = X(\omega')$, which is very reminiscent of Kantorovich's relaxation of the optimal transport problem. We come back on this kind of interpretation in the next Section.

\begin{Rem}
In the choice of super-differential we made, everything could have also been true by using not only a coupling between $µ'$ and $µ$ but a coupling $\Gamma(dx,dy,dz)$ between $µ'(dy)$ and $µ(dx)\psi(x,dz)$. Such that we could have said that $\psi \in \partial^+U(µ)$ if for all $\Gamma \in \mathcal{P}(\T^d\times \T^d\times \R^d)$ such that $(\pi_1,\pi_3)_{\#}\Gamma (dx,dz) = µ(dx)\psi(x,dz)$ and $(\pi_2)_{\#}\Gamma(dy) = µ'(dy)$, it holds that
\[
U(µ') \leq U(µ) + \int_{\T^{2d}\times \R^d}z\cdot(y-x)\Gamma(dx,dy,dz) + o\left(\left(\int_{\T^{2d}\times\R^d}|x-y|^2\Gamma(dx,dy,dz)\right)^{\frac 12}\right).
\]
\end{Rem}

 \section{Bounds on the value of the stochastic optimal transport problem near the singularity}\label{sec:sing}
We start by defining properly the value function formally introduced in Section \ref{sec:sot}. We then prove precise estimates on the behaviour of the value function $U$ near the singularity at terminal time.

\subsection{Definition of the value function}
\ca
In this section we mainly focus on the value function of the deterministic problem introduced formally in \eqref{defu}. We shall consider a non-negative cost function $L : \T^d \times \R^d \to \R$, on which assumptions shall be made later on depending on the framework which we study.

 The main difficulty in defining the value function lies in the definition of the set on which the infimum is taken in \eqref{defu}. Indeed, without regularity constraints on $\alpha$ and $m$, it is not clear how to evaluate the derivative of the product. Furthermore, $\alpha$ and $m$ have to be such that the integral which yields the cost is indeed well defined. These difficulties make it difficult to talk about $\alpha$ as the control and about $m$ as the state, as given a control, it is not clear how to define the state, as multiple solutions to the continuity equation can exist.

 In order to address this issue, in \citep{benamou2000computational}, Benamou and Brenier introduced a reformulation of the problem \eqref{benamoubrenier} into
 \[
 \inf_{m,E} \int_0^T\int_{\mathbb{T}^d}L\left(x,\frac {E_t}{m_t}\right)m_t(dx)dt,
 \]
 under the constraint that $(m,E)$ solves in the weak sense
 \be\label{eq:mE}
 \ba
 \partial_t m + \text{div}(E) = 0 \text{ in } (0,T)\times \T^d,\\
 m(0) = µ \text{ and } m(T) = \nu,
 \ea
 \ee
 and where $L\left(x,\frac Em\right)$ is set to $+\infty$ as soon as $E <<m$ is not satisfied. This (fruitful) approach allows to solve the problem of the singularity of the product $\alpha m$. However, we claim that we can introduce another way to evaluate the cost of the trajectory given by the solution of \eqref{eq:mE}, which we believe turns out to be simpler to interpret.\\
 
 Our main idea consists in saying that different "controls" can give the same evolution of the state but should yield different costs. To illustrate this, consider the following situation. The cost $L$ is simply given by $L(x,p) = |p|^2$. The initial state and the terminal constraint are both equal to $m_0$, the uniform probability measure on $\T^d$. Consider now the optimal control which consists in choosing $\alpha \equiv 0$. The associated cost is clearly $0$. Consider now the inefficient and formal control which consists in assigning to each particles, or element of mass, a constant speed chosen uniformly in the ball $B(0_{\R^d},1)$ and independently from one another. Clearly, by a symmetry argument, such a control is also admissible and also induces a constant state. However, it is very tempting to say that its cost is positive. \\
 
 To make this heuristic more precise, we introduce a problem in which the "control" is now a measurable function $\psi : [0,T] \to (\T^d \to \mprd)$. The measure $\psi_s(x,dz)$ is then interpreted as the repartition of speed we provide to the elements of mass located at $x$ at time $s$. We want to consider a "state" which is given as a solution of
 \be\label{eq:958}
 \partial_t m_t + \text{div}_x\left(\int_{\R^d} zm_t(dx)\psi_t(x,dz)  \right) = 0 \text{ in } (0,T) \times \T^d.
 \ee
The previous equation is the natural PDE to characterize the density of particles evolving with a repartition of speed $\psi$. We then want to evaluate the cost of such a pair state/control by
 \[
 \int_0^T\int_{\T^d\times\R^d}L(x,z)\psi_t(x,dz)m_t(dx)dt.
 \]
 Note that we can set $E_t(dx) := \int_{\R^d} zm_t(dx)\psi_t(x,dz) $, in which case \eqref{eq:958} is of the form of \eqref{eq:mE}. We can also set $\alpha_t(x) = \int_{\R^d}z\psi_t(x,dz)$ to realize that \eqref{eq:958} has the exact form of the usual continuity equation. In fact we have not changed the admissible trajectories but rather how to evaluate their cost. We are now ready to define properly the value functions.\\

The value function of the deterministic problem $U_{det} : [0,T)\times \mptd^2\to \R$ is defined, for $t < T, µ,\nu \in \mptd$, by
\be\label{defUdet}
U_{det}(t,µ,\nu) = \inf_{(\psi,m)} \int_t^T \int_{\T^d\times \R^d}L(x,z)\psi_s(x,dz)m_s(dx)ds,
\ee
where the infimum is taken over all pairs $(\psi,m)$ such that 
\begin{itemize}
\item $m \in \mathcal{C}([t,T],\mptd)$, and $m_t = µ, m_T = \nu$
\item $\psi : [t,T]\times \T^d \to \mprd$ is a measurable map.
\item The pair $(\psi,m)$ satisfies \eqref{eq:958} in the weak sense, i.e. for all $\varphi \in \mathcal{C}^1([t,T]\times \T^d,\R)$
\[
\int_{\T^d}\varphi(T,x)\nu(dx) - \int_{\T^d}\varphi(t,x)\mu(dx) = \int_t^T\int_{\T^d}(\partial_t\varphi(s,x) + \int_{\R^d}z\psi_s(x,dz)\cdot \nabla_x \varphi(s,x))m_s(dx)ds.
\]
\end{itemize}
We denote by $\mathcal{A}dm(t,µ,\nu)$ the set of such pairs.\\

 \bigskip
 
 Concerning the value function of the stochastic optimal transport problem, recall that we have fixed a filtered probabilistic space $(\Omega, \mathcal{A}, (\mathcal{F}_t)_{t \geq 0}, \mathbb{P})$ and a Markovian, $\mptd$ valued process $(\nu_t)_{t \geq 0}$. The value function $U : [0,T)\times \mptd^2$ is defined for $t< T, µ,\nu \in \mptd$, by
 \be\label{defUgood}
U(t,µ,\nu) = \inf_{(\psi,m)}\mathbb{E}\left[ \int_t^T \int_{\T^d\times \R^d}L(x,z)\psi_s(x,dz)m_s(dx)ds \big| \nu_t = \nu \right],
\ee
 where the infimum is taken over all pairs of random variables $(\psi,m): \Omega\to \cup_{\nu'}\mathcal{A}dm(t,µ,\nu')$ which are adapted to the filtration $(\mathcal{F}_s)_{s \geq t}$ and which are such that, $\mathbb{P}$-almost surely, on the event $\{\nu_t = \nu\}$, $m_T = \nu_T$. We denote this set $\mathcal{A}dm^{sto}(t,µ,\nu)$.
 \begin{Rem}
 The definition is exactly similar in the case in which we can make a reduction of variable by replacing $\nu$ by $w$.
 \end{Rem}

 \subsection{On the choice of the cost functional}
 We explain here, on three simple examples, the effect of the growth of the cost functional on the type of behavior we may expect near $t=T$, in the case of a deterministic problem. Such type of behaviors are well known in the optimal control theory and we shall pass through those examples quite rapidly.
 \subsubsection{Cost functional with linear growth}\label{sec:costlinear}
 Assume that the cost function $L$ is given by
 \[
 L(x,\alpha) :=|\alpha| .
 \]
 In this context, if we are concerned with \eqref{hjb1}, observe that for any $t \in [0,T]$
 \[
 U_{det}(t,µ,\nu) = U_{det}(T-1,µ,\nu) = W_1(µ,\nu).
 \]
 Indeed, for any $(t,µ,\nu)$, take an admissible pair $(\psi,m)$ in $\mathcal{A}dm(t,µ,\nu)$ and consider $\eps < T-t$. Remark that the pair $(\psi',m')$ defined by
 \[
 \ba
 \psi'_s = \frac{T-t}{T-t-\eps}Id_{\#}\psi_{\phi(s)},\\
m'_s = m_{\phi(s)},
 \ea
 \]
 where $\phi(s) = \frac{T-t}{T-t-\eps}(s-t + \eps) + t$, belongs to $\mathcal{A}dm(t+\eps,µ,\nu)$ and that moreover they have the same cost. This implies that $U_{det}(t,µ,\nu) \geq U_{det}(t+ \eps,µ,\nu)$. The inverse construction yields the opposite inequality.\\
 
 In this situation, the cost is sufficiently low for high controls to allow the value to be bounded uniformly in time. The state constraint is then very easily achieved and there is no singularity at the terminal time.
 
 If such situations may present interests in themselves, we believe that from a modeling perspective, they are not the most interesting ones as the problem of the controller does not get harder as the remaining time shortens. It is not even cleat it depends on time. We do not detail it too much but in such situations the randomness of the final target somehow disappear as the we can just wait for the final time to reach instantly the final target.
 
 \subsubsection{The case of bounded controls}
 Somehow opposite to the previous situation is the case in which $L$ is given by
 \[
 L(x,\alpha):= \begin{cases} 0 \text{ if } |\alpha| \leq 1, \\ + \infty \text{ else.} \end{cases}
   \]
 In this situation, the constrained optimization problem is not necessary controllable and the associated value can be infinite for $t \in [0,T)$. Indeed consider for instance $U_{det}(T-\epsilon,\delta_x,\delta_y)$ for $|x-y| > \epsilon$. If such cases present a lot of interest in themselves, they do not in this case in which the final density is constrained. Furthermore, if we were to replace this constraint with a bounded terminal cost, then the study of the associated HJB equation would be rather classical and fall in the scope of the previous section.

 \subsubsection{Cost functionals which are powers of distances}
 We consider here the cases in which $L$ is given, for $k \geq 1$, by 
 \be\label{power}
L(x,\alpha) := |\alpha|^k.
 \ee
 A simple change of variable yields that in this situation,
 \be\label{eqUt}
 U_{det}(t,µ,\nu) = \frac{U_{det}(T-1,\mu,\nu)}{(T-t)^{k-1}} = \frac{W_k(µ,\nu)^k}{(T-t)^{k-1}}.
 \ee
This type of behaviour is the one we are interested in, hence we shall make assumptions to control the cost function $L$ with powers of $\alpha$. Furthermore, in view of Alfonsi and Jourdain \citep{alfonsi} (which focuses on the case $k =2$), such a function $U$ is not smooth in neither $\mu$ nor $\nu$. This justifies in particular the use of the notion of viscosity solutions introduced in Section \ref{sec:hjb}.

 \subsection{Controllability of the stochastic problems and $L^{\infty}$ bounds of the value functions}
 We now provide, by means of controllability bounds, estimates on the value functions for stochastic optimal transport problems, near the final time $t = T$. In the previous section, we recalled that, as soon as the cost $L$ satisfies for some $ k \geq 1, C > 0$ and for all $x\in \T^d, p \in \R^d$
 \be\label{eq:1130}
 0\leq L(x,\alpha) \leq C (1 + |\alpha|^k),
 \ee
 then the value of the deterministic problem is bounded. In this section, we explain how we can compare the value function $U$ defined in \eqref{defUgood} with the value function $U_{det}$ of the deterministic problem, defined in \eqref{defUdet}. Define the function $\omega : [0,T) \to \R$ by
 \be\label{defomega}
 \omega(t) = \sup_{s\leq t,µ,\nu}U_{det}(s,µ,\nu).
 \ee
 In the two cases that follow, we are going to make the following assumption on $L$.
 
\begin{hyp}\label{hyp:l}
For all $t \in [0,T), \omega(t) < \infty$.
\end{hyp}
\begin{Rem}
Note that the assumptions on $L$ are here made in the previous (mild) Hypothesis, which is for instance satisfied if \eqref{eq:1130} holds.
\end{Rem}
 \subsubsection{The case of jumps}
Assume that the target process $(\nu_t)_{t \geq 0}$ is driven by jumps, which happen at Poisson times associated with the intensity $\lambda :[0,T]\to \R_+$, and which are described by the operator $\mathcal{T}:\mptd \to \mptd$. We can prove the following.
 
 \begin{Prop}\label{prop:behavjumps}
Under Hypothesis \ref{hyp:l}, assume that there exist $C >0$ and $\gamma > -1$ such that for $T-t \leq C^{-1}$
\be\label{hypgrowth}
 \lambda(t)\omega(t) \leq C (T-t)^{\gamma},
\ee
 and
\be\label{hyplambdatec}
 \lambda(t) \leq C (T-t)^{-1}\int_t^T\lambda.
\ee
Then there exists a continuous function $\beta : [0,T] \to \mathbb{R}_+$, such that
 \[
 \sup_{µ,\nu\in \mptd}|U(t,µ,\nu) - U_{det}(t,µ,\nu)| \leq \beta(t) \underset{t \to T}{\longrightarrow} 0.
 \]
 \end{Prop}
 \begin{Rem}
We comment the hypotheses of the result.
\begin{itemize}
\item The assumption \eqref{hyplambdatec} is purely technical, it is verified by any function such that $\lambda(t) \sim C (T-t)^{\alpha}$ for any $C> 0, \alpha > -1$. However it is not automatically verified, as for instance $\lambda(t) = \frac{d}{dt}(e^{-(T-t)^{-2}})$ does not verify it. We do not know wether this can be removed.
\item The requirement \eqref{hypgrowth} is quite important in our proof. This assumption yields an integrability condition on the product $\lambda \omega$. Such an integrability condition is crucial. Note for example that if $\lambda$ is constant, then we require (among other things), that $\omega \in L^1_{loc}$, which is not the case for a quadratic cost. We show an example of a situation where bounds on $U$ does not exist if this integrability fails. 
\end{itemize}
 \end{Rem}
 \begin{Rem}
 Note that no assumption is made on $\mathcal{T}$ in this result, in particular, the result still holds true if $\mathcal{T}$ depends also on $t$ and $µ$. This is due to the fact that $\mptd$ is compact. If the problem were to transport elements of $\mathcal{P}(\R^d)$, then some assumptions should be made on $\mathcal{T}$, namely on its growth.
 \end{Rem}

 \begin{proof}
 We argue first as if the infimum in the deterministic problem is always reached. Notice first that if $\lambda = 0$ in $L^1((0,\kappa),\R_+)$ for some $\kappa > 0$, then the results holds true trivially. Hence, we focus here on the case 
 \[
 \forall t > 0, \int_0^t \lambda(s)ds > 0.
 \]
Consider the problem starting in $µ \in \mptd$ at time $t$ and where the target process is equal to $\nu$ at $t$. Let $n$ be the (random) number of jumps in $[t,T]$ and consider the sequence $\tau_0 = t < \tau_1 < \tau_2 < ... < \tau_n \leq T$ of random times at which the target process jumps. Note that this sequence is finite almost surely, possibly empty and that the event $\{\tau_n = T\}$ shall be ignored since it happens with probability $0$. Consider the random pais $(\psi_s,m_s)_{s \in [t,T]}$, given by
 \begin{itemize}
 \item For $1\leq i \leq n$, $s \in (t_{i-1}, t_i)$, $(\psi_s,m_s)$ is equal (up to a change of time) to a minimum in the problem $U_{det}(t_{i-1},m_{t_{i-1}},\mathcal{T}^{i-1}\nu)$.
 \item For $s \in (t_n,T)$, $(\psi_s,m_s)$ is given through a minimum in $U_{det}(t_n,m_{t_n},\mathcal{T}^n\nu)$.
 \end{itemize}
 Such a pair is clearly admissible. We now estimate its cost.
 \[
 \begin{aligned}
& \mathbb{E}\left[ \int_t^T \int_{\mathbb{T}^d\times \R^d}L(x,z)\psi_s(x,dz)m_s(dx)ds\right]\\
 &= \mathbb{P}(n > 0)\mathbb{E}\bigg[   \int_{t_n}^T\int_{\T^d\times \R^d}L(x,z)\psi_s(x,dz)µ_s(dx)ds \\
 &\quad \quad \quad \quad \quad + \sum_{i = 0}^{n-1} \int_{t_{i}}^{t_{i+1}} \int_{\mathbb{T}^d\times \R^d}L(x,z)\psi_s(x,dz)µ_s(dx)ds \bigg| n > 0\bigg] +  \mathbb{P}(n= 0)U_{det}(t,µ,\nu)\\
 & \leq \mathbb{P}(n > 0)\mathbb{E}\left[U_{det}(t_n, m_{t_n}, \mathcal{T}^n \nu)+  \sum_{i=0}^{n-1} U_{det}(t_i,m_{t_i},\mathcal{T}^i \nu) \bigg| n > 0\right] + \mathbb{P}(n= 0)U_{det}(t,µ,\nu)\\
 & \leq \mathbb{P}(n > 0)\mathbb{E}\left[ \omega(t_n) + \sum_{i=0}^{n-1} \omega(t_i)\bigg| n > 0\right] +  \mathbb{P}(n= 0)U_{det}(t,µ,\nu)\\
 & \leq \mathbb{P}(n> 0)\mathbb{E}\left[ (1 + n)\omega(t_n) | n > 0\right] + \mathbb{P}(n= 0)U_{det}(t,µ,\nu).
 \end{aligned}
 \]
 We can now compute
 \[
 \begin{aligned}
 \mathbb{E}\left[ (1 + n)\omega(t_n)| n > 0\right] = \sum_{k = 1}^{\infty}\mathbb{E}[(1+n)\omega(t_n)|n = k]\mathbb{P}( n = k| n > 0).
 \end{aligned}
 \]
 Since the $(\tau_n)_{n \geq 0}$ are given by a Poisson process, we have that
 \[
 \mathbb{P}(n > 0) \mathbb{P}(n = k | n > 0) = \mathbb{P}(n = k) = \frac{\left(\int_t^T\lambda(s)ds\right)^k}{k!}e^{-\int_t^T\lambda(s)ds},
 \]
 and also that there exists $C >0$ such that for any $1\leq k < n$, the law of $\tau_k$ conditioned on $\tau_{k-1}$ has a density which is bounded by 
 \[
s \to \mathbb{1}_{s \geq \tau_{k-1}} C\frac{\lambda(s)}{\int_{\tau_{k-1}}^T\lambda}.
 \]
Hence, we can estimate
\[
\begin{aligned}
\mathbb{E}[(1+n)\omega(t_n)|n = k] &\leq C^{k} (k+1)\int_t^T\int_{t_1}^T\dots\int_{t_{k-1}}^T\omega(t_k)\lambda(t_k)\frac{dt_k}{\int_{t_{k-1}}^T\lambda}\dots \frac{\lambda(t_2)dt_2}{\int_{t_1}^T\lambda}\frac{\lambda(t_1)dt_1}{\int_t^T\lambda}\\
& \leq C^k(k+1)\int_t^T\int_{t_1}^T\dots\int_{t_{k-2}}^T(T-t_{k-1})^{\gamma +1}\frac{\lambda(t_{k-1})dt_{k-1}}{\int_{t_{k-1}}^T\lambda}\dots \frac{\lambda(t_2)dt_2}{\int_{t_1}^T\lambda}\frac{\lambda(t_1)dt_1}{\int_t^T\lambda}\\
& \leq \frac{C^{k}(k+1)(T-t)^{\gamma +1}}{\int_t^T\lambda(s)ds} 
\end{aligned}
\]
From the previous estimate, we deduce that
\[
\begin{aligned}
\mathbb{P}(n> 0)\mathbb{E}\left[ (1 + n)\omega(t_n) | n > 0\right] &\leq C\sum_{k = 1}^{\infty}\frac{C^k(k+1)}{k!} e^{-\int_t^T\lambda}\left(\int_t^T\lambda\right)^{k-1}(T-t)^{\gamma +1}\\
& \leq C (T-t)^{ \gamma+1}.
\end{aligned}
\]
We can compute
\[
(1 - \mathbb{P}(n = 0))U_{\det}(t,µ,\nu) \leq C \int_t^T\omega(s)\lambda(s)ds \leq C (T-t)^{ \gamma+1}.
\]
Hence, setting $\beta(t) = C (T-t)t^{\gamma+1}$, we deduce that 
\[
U(t,µ,\nu) \leq \mathbb{E}\left[ \int_t^T \int_{\mathbb{T}^d\times\R^d}L(x,z)\psi_s(x,dz)m_s(dx)ds\right] \leq U_{det}(t,µ,\nu) + \beta(t).
\]

 Obtaining the lower bound is easier. Indeed, for $\epsilon > 0$, consider an $\epsilon$ optimal pair $(\psi,m)$ (which exists since the value is bounded from below). It then follows that
 \[
 \begin{aligned}
 U(t,µ,\nu) &\geq \mathbb{E}\left[ \int_t^T \int_{\mathbb{T}^d\times \R^d}L(x,z)\psi_s(x,dz)m_s(dx)ds\right] - \epsilon\\
 &= \mathbb{P}(n = 0) \mathbb{E}\left[ \int_t^T \int_{\mathbb{T}^d\times \R^d}L(x,z)\psi_s(x,dz)m_s(dx)ds \bigg| n = 0\right] - \epsilon \\
 & \quad + \mathbb{P}(n > 0)\mathbb{E}\left[ \int_t^T \int_{\mathbb{T}^d\times \R^d}L(x,z)\psi_s(x,dz)m_s(dx)ds \bigg| n > 0\right]\\
 & \geq \mathbb{P}(n = 0) \mathbb{E}\left[ \int_t^T \int_{\mathbb{T}^d\times \R^d}L(x,z)\psi_s(x,dz)m_s(dx)ds \bigg| n = 0\right] - \epsilon\\
 & \geq  \mathbb{P}(n = 0)U_{det}(t,µ,\nu) - \epsilon.
 \end{aligned}
 \]
 Since $\epsilon$ is arbitrary, we deduce that the inequality also holds for $\epsilon = 0$. Hence, we deduce the lower bound
 \[
 U(t,µ,\nu) \geq U_{det}(t,µ,\nu) - \beta(t)
 \]
 following the same computation as in the part concerning the upper bound.\\
 
 We end the proof by remarking that if we are not able to consider optimal control for the deterministic problem, then considering appropriate $\epsilon'$ optimal controls yields the required estimates.
  \end{proof}

 The previous result yields in fact more than just bounds on the value function $U$. It gives a precise behaviour of the value function near $t = T$. It states that, under the standing assumptions, it behaves as $U_{det}$ near $t = T$.
 
  If the assumptions of the previous Theorem are not satisfied, then we can be in an entirely different situation. Indeed, consider the following example.
 \begin{Ex}
Assume $L(x,\alpha) = |\alpha|^2$, $\lambda$ is a constant and $\mathcal{T}\nu \ne \nu$ for some $\nu \in \mptd$. In this context, $\omega(t) = C(T-t)^{-1}$. Consider a time $t > 0$ and assume that at this time, the target process is equal to $\nu$.
By conditioning on the number of jumps occurring in the remaining time, we obtain that for any admissible pair $(\psi,m)$
\[
\begin{aligned}
U(t,µ,\nu) \geq& \mathbb{P}(n = 1)\mathbb{E}\left[\int_t^T \int_{\mathbb{T}^d\times \R^d} |z|^2\psi_s(x,dz)m_s(dx)ds \bigg| n =1\right]\\
& +\mathbb{P}(n = 0) \mathbb{E}\left[\int_t^T \int_{\mathbb{T}^d\times \R^d} |z|^2\psi_s(x,dz)m_s(dx)ds \bigg| n =0\right] .
\end{aligned}
\]
Let us denote by $\rho$ the density of the law of the jump $\tau_1$, conditioned on $\{n = 1\}$. Consider $(µ_s)_{s \in [t,T]}$, the trajectory in the event $\{n = 0\}$. By definition of the $2$-Wasserstein distance, we obtain that for any $t' \in [t,T)$
\be\label{eq1395}
\ba
 \mathbb{E}\left[\int_t^T \int_{\mathbb{T}^d\times \R^d} |z|^2\psi_s(x,dz)m_s(dx)ds \bigg| n =0\right]  &\geq  \mathbb{E}\left[\int_{t'}^T \int_{\mathbb{T}^d\times \R^d} |z|^2\psi_s(x,dz)m_s(dx)ds \bigg| n =0\right] \\
 &\geq\frac{W_2^2(µ_{t'},\nu)}{T-t'} 
 \ea
\ee
 We then compute
\[
 \mathbb{E}\left[\int_t^T \int_{\mathbb{T}^d\times \R^d} |z|^2\psi_s(x,dz)m_s(dx)ds \bigg| n =1\right] \geq \int_t^T\frac{W_2^2(µ_s, \mathcal{T}\nu)}{T-s}\rho(s)ds
\]
Integrating \eqref{eq1395} with respect to $\rho$, we deduce that
\[
\begin{aligned}
U(t,µ,\nu) \geq& \mathbb{P}(n = 0)\int_t^T\frac{W_2^2(µ_{s},\nu)}{T-s}\rho(s)ds\\
& +\mathbb{P}(n = 1)\int_t^T\frac{W_2^2(µ_{s},\mathcal{T}\nu)}{T-s} \rho(s)ds\\
\geq &\mathbb{P}(n = 1)\int_t^T\frac{W_2^2(µ_s,\nu) + W_2^2(µ_{s},\mathcal{T}\nu)}{T-s} \rho(s)ds,
\end{aligned}
\]
if $t$ is sufficiently small so that $\mathbb{P}(n= 1) \leq \mathbb{P}(n=0)$. The right hand side of the previous inequality is equal to $+ \infty$ since $\nu \ne \mathcal{T}\nu$. Hence for any $µ \in \mptd$, $U(t,µ,\nu) = + \infty$.
 \end{Ex}

 This last example hints that there is a strong dichotomy : either $U$ is infinite in all the points $\nu$ such that $\mathcal{T}\nu \ne \nu$, or either it behaves quite similarly as $U_{det}$.

 \subsubsection{The case of the target pushed by a diffusion}
Consider now that $(\nu_t)_{t \geq 0}$ is given by $\nu_t = (\tau_{W_t})_{\#}\nu$ for $\nu \in \mptd$ and $(W_t)_{t \geq 0}$ the strong solution of
\[
dW_t = \sigma(t)dB_t \text{ for } t > 0,
\]
with initial condition $W_0 = 0$, where $\sigma : [0,T)\to \R$ is a smooth bounded function and $(B_t)_{t \geq 0}$ is a standard Brownian motion. We also assume that there exists $C > 0$ such that for all $x\in \T^d, \alpha \in \R^d$
\[
L(x,\alpha) \leq C(1 + |\alpha|^2).
\]
Also recall that we are here interested in the value function $U$ as a function of $t,µ$ and $w\in \T^d$.

We start with the following Lemma.
\begin{Lemma}\label{lemmacompw}
Almost surely, there exists a unique $\T^d$ valued solution $(X_t)_{t \leq T}$ of
\be\label{eqXlemma}
dX_t = \frac{W_t - X_t}{T-t}dt,
\ee
given an initial condition $X_0 \in \T^d$. Almost surely, it satisfies $X_t \to W_T$ as $t \to T$.

Moreover, 
\[
\mathbb{E}\left[\int_0^T\left| \frac{dX_t}{dt}\right|^2ds\right] = \frac{\mathbb{E}[|W_0-X_0|^2]}{T}+ \int_0^T\frac{\sigma(s)^2}{T-s}ds,
\]
which possibly reads $+\infty = + \infty$.
\end{Lemma}
\begin{proof}
Let us first remark that the existence and uniqueness of the solution on $[0,T)$ is trivial. Hence we only need to show that the the limit holds as $t \to T$. Remark now that \eqref{eqXlemma} can be written
\[
dX_t = \frac{W_T - X_t}{T-t}dt + \frac{W_t - W_T}{T-t}dt.
\]
The previous leads to
\[
\frac{d |X_t-W_T|^2}{dt} = -2 \frac{|X_t-W_T|^2}{T-t} + 2\frac{(W_t - W_T)\cdot(X_t - W_T)}{T-t}.
\]
Integrating this relation yields
\[
|X_t-W_T|^2 + 2 \int_0^t\frac{|X_s - W_T|^2}{T-s}ds = 2\int_0^t\frac{W_t - W_T}{\sqrt{T-s}}\cdot\frac{X_s-W_T}{\sqrt{T-s}}ds + |X_0 - W_T|^2.
\]
From the regularity property of the Brownian motion, more precisely that, almost surely, for $t$ and $s$ sufficiently close,
\[
\forall c > 1, |B_t -B_s| \leq c \sqrt{2|t-s|\log(|t-s|^{-1})},
\]
we deduce that there exists $C> 0$ independent of $t$ such that, almost surely,
\[
|X_t-W_T|^2 + 2 \int_0^t\frac{|X_s - W_T|^2}{T-s}ds \leq C.
\]
Hence the first part of the results follows.\\

Let us now remark that 
\[
d\frac{W_t - X_t}{T-t} = \frac{dW_t}{T-t} = \frac{\sigma(t)dB_t}{T-t}.
\]
Hence, we deduce that
\be\label{eq:dxdt}
\frac{dX_t}{dt} = \frac{W_t - X_t}{T-t} = \frac{W_0-X_0}{T} + \int_0^t\frac{\sigma(s)}{T-s}dB_s.
\ee
From which follows
\[
\mathbb{E}\left[\int_0^T\left| \frac{dX_t}{dt}\right|^2ds\right] =  \frac{\mathbb{E}[|W_0-X_0|^2]}{T}+ \int_0^T\int_0^t\frac{\sigma(s)^2}{(T-s)^2}dsdt,
\]
which yields the result.
\end{proof}
 We can now prove a controllability estimate.
\begin{Prop}\label{prop:behavmb}
Assume that there exists $C \geq 0$ such that $L$ satisfies for any $x,\alpha$
\be\label{eq:growthL}
 L(x,\alpha) \leq C(1 + |\alpha|^2),
\ee
and that
\be\label{hyp:sigma}
\sigma(t) \sim K(T-t)^{\gamma} \text{ as } t \to T,
\ee
for some $K \ne 0,\gamma > 0$. Then for all $t < T, µ \in\mptd, w \in \T^d$, $U(t,µ,w) < \infty$.
\end{Prop}
\begin{proof}
Consider $(t,µ,w) \in (0,\infty)\times \mptd \times \mathbb{T}^d.$ Take $\delta > 0$ such that $t + \delta < T$ and define $T_0 = T -( t + \delta)$. Denote by $(W_s)_{s \geq t}$ the strong solution of \eqref{sde} such that $W_t = w$. Consider the control $\alpha$ which transports optimally, according to the quadratic cost, $µ$ into $\nu$ in a time $\delta$. Consider the process $(X_s)_{s \geq t + \delta}$ defined by
\[
dX_s = \frac{W_s - X_s}{T-s}ds,
\]
with initial condition $X_{t+ \delta} = 0$. We can now build an admissible control by setting $\psi_s(x,dz) = \delta_{X_s}$.
Thanks to Lemma \ref{lemmacompw}, this control is admissible. Moreover, thanks to the same result, we have the trivial estimate
\[
U(t,µ,w) \leq C(T-t) + C\frac{W_2^2(µ,\nu)}{\delta} + C  \frac{\mathbb{E}[|w_{t +\delta}|^2]}{T_0}+ \int_{t + \delta}^{T}\frac{\sigma(s)^2}{T-s}ds,
\]
which is finite thanks to \eqref{hyp:sigma}.
\end{proof}
We can easily extend the previous result (with a different assumption on $\gamma$) to more general cost functionals $L$.
\begin{Cor}\label{cor:sigmak}
Assume that there exists $C \geq 0, K \ne 0, \gamma > \frac 12$ and $k \geq 1$ such that for any $x,\alpha$,
\[
 L(x,\alpha) \leq C(1 + |\alpha|^k),
\]
\[
\sigma(t) \sim K(T-t)^{\gamma} \text{ as } t \to T.
\]
Then for all $t < T, µ \in\mptd, w \in \T^d$, $U(t,µ,w) < \infty$.
\end{Cor}
\begin{proof}
The argument is similar to the previous one. Consider $(X_t)_{t \geq 0}$ given in \eqref{lemmacompw}. Observe that \eqref{eq:dxdt} still holds. Then, thanks to the Burkholder-Davis-Gundy inequality, we obtain that
\[
\mathbb{E}\left[\int_0^T\left| \frac{dX_t}{dt}\right|^kds\right] =  C\frac{\mathbb{E}[|W_0-X_0|^k]}{T^{k-1}}+ C\int_0^T\frac{\sigma(t)^2}{(T-t)^2}dt.
\]
Hence, the result follows from the same argument as in Proposition \ref{prop:behavmb}.
\end{proof}

We now prove a refinement of the previous estimate which yields a more precise result for the behaviour of $U$ near $t= T$.
\begin{Prop}\label{prop:1532}
Under Hypothesis \ref{hyp:l} and the assumptions and notations of Proposition \ref{prop:behavmb}, assume furthermore that $\gamma > \frac 12$, and that for any $t < s < T$
\be\label{hypdt}
U_{det}(s,µ,\nu) - U_{det}(t,µ,\nu) \leq (T-s)^{-2}(s-t).
\ee
Then, it holds that
\[
U(t,µ,w) \leq U_{det}(t,µ,(\tau_w)_{\#}\nu) + \beta(t),
\]
for a continuous function $\beta$ such that $\beta(t) \to 0$ as $t \to T$. Moreover, if $L$ is convex in $\alpha$ we always have
\[
U_{det}(t,µ,(\tau_w)_{\#}\nu) \leq U(t,µ,w).
\]
\end{Prop}
\begin{proof}
The proof is similar to the previous one. Consider  $(t,µ,w) \in (0,\infty)\times \mptd \times \mathbb{T}^d.$ Take $\delta(t) $ such that $t + \delta(t) < T$ and define $T_0(t) = T - (t + \delta(t))$. The function $\delta$ is to be chosen later on.

 Denote by $(W_s)_{s \geq t}$ the strong solution of \eqref{sde} such that $W_t = w$. Consider an optimal trajectory, according to the deterministic problem, which transports $µ$ into $(\tau_w)_{\#}\nu$ in a time $\delta(t)$. Consider the process $(X_s)_{s \geq t + \delta(t)}$ defined by
\[
dX_s = \frac{W_s - X_s}{T-s}ds,
\]
with initial condition $X_{t+ \delta(t)} = w$. Consider now the control which consists in playing the first trajectory in the time $\delta(t)$ and then to translate the state with speed $\frac{dX_s}{ds}$ afterwards. Thanks to Lemma \ref{lemmacompw}, this control is admissible. Moreover, thanks to the same result, we can estimate
\be\label{eq567}
U(t,µ,w) \leq U_{det}(T-\delta(t),µ, (\tau_w)_{\#}\nu) + C  \frac{\mathbb{E}[|w_{t+\delta(t)} - w|^2]}{T- (t+ \delta(t))}+ \int_{t + \delta(t)}^{T}\frac{\sigma(s)^2}{T-s}ds.
\ee
Remark that
\[
\mathbb{E}[|w_{t+\delta(t)} - w|^2] = \int_t^{t + \delta(t)}\sigma^2(s)ds \leq K(T-t)^{2\gamma}\delta(t),
\]
where the inequality holds for $t$ close to $T$. Let us now set $\delta(t) = (T-t) - (T-t)^{\theta}$ for some $\theta > 1$. Coming back to \eqref{eq567}, we obtain
\[
U(t,µ,w) \leq U_{det}(t + (T-t)^{\theta},µ, (\tau_w)_{\#}\nu) + C  \frac{K(T-t)^{2\gamma}\delta(t)}{(T-t)^{\theta}}+ \int_{t + \delta(t)}^{T}\frac{\sigma(s)^2}{T-s}ds
\]
The last term vanishes as soon as $\gamma > 0$, the second to last term vanishes as soon as $1 + 2\gamma > \theta$. It then remains to estimate the first term of the right hand side. Namely, we are interested in the difference
\[
U_{det}(t + (T-t)^{\theta},µ, (\tau_w)_{\#}\nu) - U_{det}(t,µ,(\tau_w)_{\#}\nu).
\]
From \eqref{hypdt}, we can bound this difference by $C(T-t)^{\theta - 2}$ and we deduce finally that
\[
U(t,µ,w) \leq U_{det}(t,µ, (\tau_w)_{\#}\nu) + C(T-t)^{\theta-2} +C  \frac{K(T-t)^{2\gamma}\delta(t)}{(T-t)^{\theta}}+ \int_{t + \delta(t)}^{T}\frac{\sigma(s)^2}{T-s}ds,
\]
which yields the required estimate when $\theta \in (2,1 + 2 \gamma)$.\\

Finally, let us remark that the estimate $U(t,µ,w) \geq U_{det}(t,µ,(\tau_w)_{\#}\nu)$ is simply a consequence of the convexity of $L$ in $\alpha$. Indeed, fix $t < T, µ \in \mptd, w \in \T^d$ and consider any (stochastic) admissible pair $(\psi,m)$. As we already mentioned at the beginning of this Section, we can always consider an associated pair of measures $(m,E)$ solution of 
\[
\partial_t m + \text{div}(E) = 0 \text{ in } (t,T)\times \T^d,
\]
which satisfies $m_t = µ,m_T = (\tau_{W_T})_{\#}\nu$. Defining $M_t = \mathbb{E}[m_t]$ and $K_t = \mathbb{E}[E_t]$, we obtain that
\[
\partial_t M + \text{div}(K) = 0 \text{ in } (t,T)\times \T^d,
\]
together with $M_t = µ, M_T = (\tau_w)_{\#}\nu$. From this we deduce that
\[
\ba
\mathbb{E}\left[\int_t^T \int_{\T^d\times \R^d}L(x,z)\psi_s(x,dz)m_s(dx)ds\right] &\geq \mathbb{E}\left[\int_t^T \int_{\T^d}L\left(x,\frac{E_s}{m_s}\right)m_s(dx)ds\right]\\
& \geq \int_t^T\int_{T^d}L\left(x,\frac{K_s}{M_s}\right)m_s(dx)ds\\
& \geq U_{det}(t,µ,(\tau_w)_{\#}\nu),
\ea
\]
from which the result follows.

\end{proof}
\begin{Rem}
Although it is not the main objective of the present work, it would be interesting to lower the assumption $\gamma > \frac 12$ into $\gamma >0$ so that this result is similar to the one we present for the case of jumps, namely that the controllability directly yields a precise behaviour of $U$ near $t = T$.
\end{Rem}

\begin{Rem}
The assumption \eqref{hypdt} is a bound on the time derivative of $U_{det}$. It is verified in the case $L(x,p) = |p|^2$ for instance. Since $L$ does not depend explicitly on $t$, it can be verified by a change of variable in time as soon as $D_{\alpha}L$ has linear growth in $|\alpha|$, see for instance the computation of the next section.
\end{Rem}
 \begin{Rem}\label{rem:BDG}
 This result can also be adapted to more general cost functions $L$ as in Corollary \ref{cor:sigmak}, once again by using the Burkholder-Davis-Gundy inequality.
 \end{Rem}

\section{Continuity of the value function and viscosity solutions properties}\label{sec:cont}
Before proving that the value functions studied in the previous section are indeed viscosity solutions of the associated HJB equations, we start by proving some continuity estimates on these value functions. We assume in this section that
\be\label{eq:hyppoly}
\exists k > 1, C >0, \forall x\in \T^d,\alpha \in \R^d, \quad 0 \leq L(x,\alpha) \leq C(1+|\alpha|^k).
\ee
This condition is enough to ensure that the value of the deterministic problem is finite everywhere on $\{t < T\}$.

\subsection{Continuity estimates}
The continuity estimates we are going to provide rely on the controllability of the problem. Namely our strategy of proof consists in remarking that, if the problem if sufficiently controllable, then, with closed initial conditions, we can reduce one case to the other. In this section, we focus on the value function associated to the evolution described in section \ref{sec:reduw}. We comment later on on the case of jumps. Thus the main of object at interest here is $U$ defined by
\be\label{defUw}
U(t,µ,w) = \inf_{(\psi,m) \in\mathcal{A}dm^{sto}(t,µ,w) }\mathbb{E}\left[ \int_t^T \int_{\T^d\times \R^d}L(x,z)\psi_s(x,dz)m_s(dx)ds \big| W_t = w \right],
\ee
where we recall that $(W_s)_{s \geq 0}$ is the strong solution of 
\[
dW_t = \sigma(t) dB_t,
\] 
with initial condition $0$ for $(B_t)_{t \geq 0}$ a standard Brownian motion on $(\Omega, \mathcal{A},\mathbb{P})$. Recall that the terminal constraint is given by $m_T = (\tau_{W_T})_{\#}\nu$. Thanks to Corollary \ref{cor:sigmak}, to ensure that $U$ is finite, in addition to \eqref{eq:hyppoly}, we also assume that
\be\label{eq:hypsigma}
\exists K \ne 0, \gamma > \frac 12, \quad \sigma(t) \sim_{t\sim T} K(T-t)^\gamma.
\ee

 The following results depend on various assumptions on the cost function $L$ that we state progressively to highlight the effect of each one of them.\\

We begin with the following Lemma which states that the set of controls which can be reduced by a density argument.
\begin{Lemma}\label{lemma:density}
For $t < T, µ , \nu \in \mptd$ and $\eps < T-t$, define 
\[
\mathcal{A}dm^{sto}_\eps(t,µ,w) := \{ (\psi,m) \in \mathcal{A}dm^{sto}(t,µ,w) | \forall s \leq t + \eps, \psi_s = \delta_0\}.
\]
Assume that \eqref{eq:hyppoly} and \eqref{eq:hypsigma} hold and that
\be\label{hyp:dpl}
\exists C > 0, \forall x,z, \quad |D_pL(x,z)||z|\leq C (1 + L(x,z)).
\ee
Then, $U$ is also given by
\be\label{defUeps}
U(t,µ,w) = \inf_{\eps > 0,(\psi,m) \in\mathcal{A}dm^{sto}_\eps(t,µ,\nu) }\mathbb{E}\left[ \int_t^T \int_{\T^d\times \R^d}L(x,z)\psi_t(x,dz)m_t(dx)dt \big| w_t = w \right].
\ee
Moreover, for all $\eps > 0$ there exists an $\eps$ optimal control in $\mathcal{A}dm^{sto}_{C\eps}(t,µ,w)$ with $C >0$ depending only $U(t,µ,w)$ and $T-t$.
\end{Lemma}
\begin{proof}
Of course, $U$ is always smaller than the right hand side of \eqref{defUeps}. We thus prove the reverse inequality. Consider $n \geq 1, \eps < T-t$ and a $n^{-1}$ optimal control $(\psi,m)$ for $U$. Using the same change of variable as in Section \ref{sec:costlinear}, we remark that $(\psi',m')$ defined for $s \geq t + \eps$ by 
 \[
 \ba
 \psi'_s = \frac{T-t}{T-t-\eps}Id_{\#}\psi_{\phi(s)}\\
m'_s = m_{\phi(s)},
 \ea
 \]
where $\phi(s) = \frac{T-t}{T-t-\eps}(s-t + \eps) + t$ and by $\psi'_s = 0, m'_s = µ$ for $s \in [t,t+ \eps]$ belongs to $\mathcal{A}dm_\eps(t,µ,\nu)$. We now evaluate
\[
\ba
\mathbb{E}\left[ \int_t^T \int_{\T^d\times \R^d}L(x,z)\psi'_s(x,dz)m'_s(dx)ds \right] &= \mathbb{E}\left[ \int_{t+\eps}^T \int_{\T^d\times \R^d}L\left(x,\phi'z\right)\psi_{\phi(s)}(x,dz)m_{\phi(s)}(dx)ds\right]\\
=& (\phi')^{-1}\mathbb{E}\left[ \int_{t}^T \int_{\T^d\times \R^d}L\left(x,\phi'z\right)\psi_{s}(x,dz)m_{s}(dx)ds\right]\\
 =& (\phi')^{-1}\mathbb{E}\left[ \int_{t}^T \int_{\T^d\times \R^d}L\left(x,z\right)\psi_{s}(x,dz)m_{s}(dx)ds\right] + \\
 + &(\phi')^{-1}\mathbb{E}\left[ \int_{t}^T \int_{\T^d\times \R^d}(L\left(x,\phi'z\right) - L(x,z))\psi_{s}(x,dz)m_{s}(dx)ds\right]\\
  \leq& (\phi')^{-1}(U(t,µ,w) +n^{-1} ) + \\
  &+ \frac{\phi' -1}{\phi'}C\mathbb{E}\left[ \int_{t}^T \int_{\T^d\times \R^d}( 1 +L(x,z))\psi_{s}(x,dz)m_{s}(dx)ds\right],
\ea
\]
where we have used \eqref{hyp:dpl} and the $n^{-1}$ optimality of $(\psi,m)$ in the last inequality. Since $\phi' \to 1$ as $\eps \to 0$, we thus on the first part of the claim. The second part can be observed simply by remarking that $$|\phi'- 1| \leq C \eps$$ as $\eps \to 0$, for $C$ depending only on $T-t$.
\end{proof}

We are now ready to prove the following.
\begin{Lemma}\label{lemma:contmu}
Assume that \eqref{eq:hyppoly}, \eqref{eq:hypsigma} and \eqref{hyp:dpl} hold. For any $t < T$ and $w \in \T^d$, there exists $C$ depending only on $\sup_{µ \in \mptd}U(t,\mu,w)$ and $T-t$ such that
\[
\forall µ,µ' \in \mptd, \quad|U(t,µ,w) - U(t,µ',w)|\leq C W_k(µ,µ').
\]
\end{Lemma}
\begin{proof}
Take $t < T, w \in \T^d$ and $µ,µ' \in \mptd$. We want to show that $U(t,µ',w) \to U(t,µ,w)$ as $µ' \to µ$. Thanks to Lemma \ref{lemma:density}, consider an $\eps$ optimal control $(\psi,m) \in \mathcal{A}^r_{C\eps}(t,µ,w)$ for $U(t,µ,w)$ for some $\eps > 0$. Consider now the pair $(\psi',m')$ defined by 
\[
(\psi'_s,m'_s) = (\psi_s,m_s) \text{ for } s \in [t+\eps,T),
\]
and $(\psi',m')_{s \in [t,t+ \eps]}$ corresponds to an optimal trajectory for the deterministic optimal transport of $µ'$ toward $µ$ in time $\eps$ for the cost $L(x,z) = |z|^k$. We can then estimate
\[
U(t,µ',w) \leq U(t,µ,w) + \eps + C^{1-k}( W_k^k(µ,µ')\eps^{1-k}). 
\]
Taking $\eps = W_k(µ,µ')$ yields the required result since $µ'$ and $µ$ are arbitrary in $\mptd$.
\end{proof}

We now show continuity of the value function with respect to $w$.
\begin{Lemma}
Assume that \eqref{eq:hyppoly}, \eqref{eq:hypsigma} and \eqref{hyp:dpl} hold and that
\be\label{hyp:dxl}
\exists C >0, \forall x, z, \quad |D_xL(x,z)| \leq C ( 1 + L(x,z)).
\ee
Then for all $t < T, µ\in \mptd$, there exists $C>0$ depending only on $\sup_{w \in \T^d}U(t,\mu,w)$ and $T-t$ such that
\[
\forall w,w' \in \mptd, \quad|U(t,µ,w) - U(t,µ,w')|\leq C |w-w'|.
\]
\end{Lemma} 
\begin{proof}
Take $t <T$, $µ \in \mptd$ and $w,w' \in \T^d$. Denote $\bar w = \frac{w- w'}{T-t} \in \R^d$. Take $\eps > 0$ and an $\eps$ optimal control $(\psi,m)$ for $U(t,µ,w)$. Consider the control $(\psi',m')$ defined by
\[
\ba
\psi'_s(x + (s-t)\bar w,dz) = (\tau_{\bar w})_{\#}\psi_s(x,dz), \quad x \in \T^d, t \leq s < T,\\
m'_s = (\tau_{(s -t)\bar w})_{\#}m_s, \quad t\leq s < T.
\ea
\]
In other words, $(\psi',m')$ corresponds to the same control as $(\psi,m)$ but with the fact that we are using in addition a uniform speed of $\bar w$. The control $(\psi',m')$ is admissible for $U(t,µ,w')$ and we can estimate
\[
\ba
U(t,µ,w') &\leq \mathbb{E}\left[ \int_{t}^T \int_{\T^d\times \R^d}L\left(x + (s-t)\bar w,\bar w + z\right)\psi_{s}(x,dz)m_{s}(dx)ds\right]\\
& \leq  \mathbb{E}\left[ \int_{t}^T \int_{\T^d\times \R^d}L\left(x + (s-t)\bar w,\bar w + z\right) - L(x,z)\psi_{s}(x,dz)m_{s}(dx)ds\right]\\
& \quad + U(t,µ,w) + \eps\\
& \leq  U(t,µ,w) + \eps + C \mathbb{E}\left[ \int_{t}^T \int_{\T^d\times \R^d}L(x,z)\psi_{s}(x,dz)m_{s}(dx)ds\right]\bar w.
\ea
\]
The result now follows by taking the limit $\epsilon \to 0$.
\end{proof}

Let us now remark that combining the proofs of the previous two lemmas, we arrive easily at the following.
\begin{Prop}\label{prop:continuitymuw}
Assume that \eqref{eq:hyppoly}, \eqref{eq:hypsigma}, \eqref{hyp:dpl} and \eqref{hyp:dxl} hold. Then, for all $t < T$, there exists $C >0$ depending only on $\sup_{µ\in \mptd,w \in \T^d}U(t,µ,w)$ and on $T-t$ such that
\[
\forall µ,µ' \in \mptd,w,w' \in \T^d, \quad|U(t,µ,w) - U(t,µ',w')|\leq C (W_k(µ,µ') + |w-w'|).
\]
\end{Prop}

We now show global continuity of the value function.

\begin{Prop}\label{prop:globalcont}
Assume that \eqref{eq:hyppoly}, \eqref{eq:hypsigma}, \eqref{hyp:dpl} and \eqref{hyp:dxl} hold. Then, for any $\eps > 0$, there exists $C> 0$ such that for all $0 \leq t,t' \leq T-\eps, µ,µ' \in \mptd, w,w' \in \T^d$,
\be\label{eq:globalcont}
|U(t,µ,w) - U(t',µ',w')| \leq C (\sqrt{|t-t'|} + W_k(µ,µ') + |w-w'|).
\ee
\end{Prop}
\begin{proof}
Consider $(t,µ,w), (t',µ',w') \in [0,T) \times \mptd \times \T^d$ and assume that $t' > t$. Considering a control for $U(t,µ,w)$ which does nothing in $[t,t']$, we obtain that
\[
U(t,µ,w) \leq \mathbb{E}[U(t',µ,W_{t'})].
\]
We deduce that
\[
U(t,µ,w) - U(t',µ,w) \leq  \mathbb{E}[U(t',µ,W_{t'}) - U(t',µ,w)],
\]
which implies
\be\label{eq:1582}
U(t,µ,w) - U(t',µ',w') \leq \mathbb{E}[U(t',µ,W_{t'}) - U(t',µ,w)] + C(W_k(µ,µ') + |w-w'|),
\ee
for some constant $C$ depending only $T - \max(t,t')$ thanks to Proposition \ref{prop:continuitymuw}.

 Thanks to Lemma \ref{lemma:density}, there exists a $C(t'-t)$ optimal control for $U(t,µ,w)$ which does nothing in the time interval $[t,t']$, for $C$ depending only on $U(t,µ,w)$. Observe that, $\mathbb{P}$ almost surely, such a control is admissible for the problem $U(t',µ,W_{t'})$ where $W_{t'} = w + \int_0^{t'}\sigma(s)dBs$. This remark leads to the estimate
 \[
 U(t,µ,w) \geq \mathbb{E}[U(t',µ,W_{t'})] - C(t'-t).
 \]
Arguing as above, this leads to
\[
U(t,µ,w) - U(t',µ',w') \geq - C(t'-t)+ \mathbb{E}[U(t',µ,W_{t'}) - U(t',µ,w)] - C(W_k(µ,µ') + |w-w'|),
\]
Recalling \eqref{eq:1582} and the Lipschitz continuity of $U$ in $w$, we finally obtain \eqref{eq:globalcont}.
\end{proof}
\begin{Rem}
When the target process is given by a jump process, assumptions on the operator $\mathcal{T}$ have to be made in order to establish continuity of the value function with respect to $\nu$. Nonetheless, such a regularity is less important in this case as, $\mathbb{P}$-almost surely, the trajectory $(\nu_s)_{s \geq 0}$ only takes a finite number of values.
\end{Rem}

\subsection{Viscosity solution properties of the value functions}
 We now prove that the value functions at interest are indeed viscosity solutions of the corresponding HJB equations. Since they are mainly three value functions at interest (deterministic, case of jump and the case of the diffusion), we do not provide three complete proofs, but rather establish the property of viscosity sub-solution in one case and the property of viscosity super-solution in another one. We comment on the remaining cases at the end of the section and leave to the interested reader the adaptations of the proofs we provide.\\
 
 We start by proving that the value function associated to the case of Section \ref{sec:reduw} is a sub-solution of the HJB equation \eqref{hjbw}. 
 \begin{Prop}
Assume that \eqref{eq:hyppoly}, \eqref{eq:hypsigma}, \eqref{hyp:dpl} and \eqref{hyp:dxl} hold. Then, $U$ defined in \eqref{defUw} is a viscosity sub-solution of \eqref{hjbw}.
 \end{Prop}
 \begin{proof}
Take $(t,µ,w) \in [0,T)\times \mptd\times \T^d$ and $(\theta,\psi,p,X) \in J^+U(t,µ,w)$. In all the proof that follows, to lighten notation, it is always assumed that all expectation with respect to the probability space $(\Omega,\mathcal{A},\mathbb{P})$ are taken conditionally on $\{W_t =w\}$. Consider an atomeless probabilistic space $(\Omega',\mathcal{A}',\mathbb{P}')$ and a couple $(X_t,Z)$ of random variables on $\Omega'$ such that $\mathcal{L}(X_t,Z)(dx,dz) = m_t(dx)\psi^*(x,dz)$. Consider the ordinary differential equation 
 \[
 dX_s = -D_pH(X_s,Z)ds \quad \text{ for } s \in (t,T),
 \]
 with initial condition $X_t$ at time $s = t$, which considered valued in $\T^d$. Take $\kappa \in (0,T-t)$ and for $s \in [t,\kappa]$, we introduce the disintegration $\mathcal{L}(X_s,Z) = m_s(dx)\phi_s(x,dz)$. Using the control $(m_s,\phi_s)$ in the time interval $[t,t + \kappa]$ leads to
 \[
 U(t,µ,w) \leq \mathbb{E}_{\mathbb{P}}\left[\int_t^{t+ \kappa}\int_{\T^d\times \R^d}L(x,-D_pH(x,z))\phi_s(x,dz)m_s(dx)ds + U(t+ \kappa,m_{t + \kappa},W_{t + \kappa})\right].
 \]
 Let us insist upon the fact that the sort of dynamic programming principle (DDP) we used here is standard thanks to the regularity given by Proposition \ref{prop:globalcont}. We refer to Fleming and Soner \citep{flemingsoner} for a standard presentation of DDP in finite dimension, to Claisse et al. \citep{claisse} for a more precise discussion on the topic and to Djete et al. \citep{djete} for a discussion in the case of functions of probability measures. 
 
 Using $(\theta,\psi,p,X) \in J^+U(t,µ,w)$ we obtain 
 \[
 \ba
 0 \leq &\mathbb{E}_{\mathbb{P}}\left[\int_t^{t+ \kappa}\int_{\T^d\times \R^d}L(x,-D_pH(x,z))\phi_s(x,dz)m_s(dx)ds + \mathbb{E}_{\mathbb{P}'}[Z\cdot(X_{t+\kappa}-X_t)]\right]\\
 &+ (\theta + \sigma^2(t)Tr(X))\kappa + \kappa\omega(\kappa),
 \ea
 \]
 where $\omega$ is a real continuous function such that $\omega(0) = 0$. Remark now that we can remove the expectation with respect to $(\Omega,\mathbb{P})$. Using the link between $L$ and $H$ and the definition of $\phi_s$, we obtain
  \[
 \ba
 0 \leq &\int_t^{t+ \kappa}\int_{\T^d\times \R^d}-H(x,z)\phi_s(x,dz)m_s(dx)ds+ \int_t^{t+\kappa}\mathbb{E}_{\mathbb{P}'}[Z\cdot D_pH(X_s,Z)]ds + \mathbb{E}_{\mathbb{P}'}[Z\cdot(X_{t+\kappa}-X_t)]\\
 &+ (\theta + \sigma^2(t)Tr(X))\kappa + \kappa\omega(\kappa),
 \ea
 \]
 Simplifying and rewriting the first integral in terms of $X_s$ and $Z$, we obtain
 \[
  0 \leq \int_t^{t+ \kappa}\mathbb{E}_{\mathbb{P}'}[-H(X_s,Z)]ds+ (\theta + \sigma^2(t)Tr(X))\kappa + \kappa\omega(\kappa).
 \]
 Dividing by $\kappa$ and taking the limit $\kappa \to 0$ yields the required result since $(X_s)_{s > t}$ converges uniformly, $\mathbb{P}'$-almost surely, toward $X_t$ as $s \to t$ because $Z$ is bounded.
 \end{proof}

We now pass to the more technical property of viscosity super-solution, in the case of the value of the deterministic problem. Hence, in the rest of this section, we fix $\nu \in \mptd$ and we denote $U(t,µ)=U_{det}(t,µ,\nu)$ for $U_{det}$ defined by \eqref{defUdet}. Take $\eps > 0$ and define, for $K > 0$, $(t,µ) \in [0,T-\eps]\times \mptd$
\be\label{defUK}
U_K(t,µ) = \inf_{(\psi,m)} \left\{\int_t^{T-\eps}\int_{\T^d\times \R^d}L(x,z)\psi_s(x,dz)m_s(dx)ds + U(T-\eps,m_{T-\eps})\right\},
\ee
where the infimum is taken over controls such that $\psi$ is, uniformly in $x$, supported in the ball $B(0,K)$. Note that by looking at the problem on $[0,T-\eps]$ with bounded terminal condition $U(T-\eps,\cdot)$ we do not have to worry about controllability issues related to the terminal constraint. Moreover, since the terminal condition $U(T-\eps,\cdot)$ is Lipschitz continuous, it follows easily that for any $t < T-\eps$ and $µ\in \mptd$, $\lim_{K \to \infty} U_K(t,µ) = U(t,µ)$.

Our strategy is the following: show first that for any $K>0$, $U_K$ is a viscosity super-solution of \eqref{hjb}, then prove that $U = \inf_{K >0}U_K$ is also a viscosity super-solution of \eqref{hjb}. Hence we start with the following.
\begin{Lemma}
Assume that \eqref{eq:hyppoly}, \eqref{hyp:dpl} and \eqref{hyp:dxl} hold. Then, for any $K > 0$, $U_K$ is a viscosity super-solution of the equation \eqref{hjb} on $[0,T-\eps)\times \mptd$.
\end{Lemma}
\begin{proof}
We start by remarking that, since $µ \to U(T-\eps,µ)$ is Lipschitz continuous, it follows that $U_K$ is continuous on $[0,T-\eps]\times \mptd$. For any $(\psi,m) \in \mathcal{A}dm_K$, because the support of $\psi$ is bounded, thanks to representation theorems such as Theorem 2.1 in Jimenez et al. \citep{jimenez2022dynamical}, we obtain that there exists a $\T^d$ valued random process $(X^\psi_s)_{s \in [t_n,T]}$ on an atomeless probabilistic space $(\Omega',\mathcal{A}',\mathbb{P}')$, such that, for all $s\in [t,T-\eps]$ $\mathcal{L}(X^\psi_s) = m_s$ and for almost every $s \in [t,T-\eps]$,
\[
dX^\psi_s = \int_{\R^d}z\psi_s(X^\psi_s,dz)ds.
\]
Consider now $(\theta,\psi^*) \in \partial^-U_K(t,µ)$ for some $(t,µ) \in [0,T-\eps)\times \mptd$.

Assume in a first time that there exists $(X,Z)$ on $\Omega'$ such that $\mathcal{L}(X,Z) = µ(dx)\psi^*(x,dz)$ and for any $(\psi,m)$ admissible controls, $X^\psi_t = X$.
Observe that for any $\kappa \in (0,T-\eps-t)$
\[
U_K(t,µ) = \inf_{(\psi,m)}\left\{\int_t^{t + \kappa}\int_{\T^d\times \R^d}L(x,z)\psi_s(x,dz)m_s(dx)ds + U_K(t + \kappa,m_{t + \kappa})\right\}.
\]
This yields, using $(\theta,\psi^*) \in \partial^-U_K(t,µ)$,
\[
\ba
0 \geq \inf_{(\psi,m)} \bigg\{&\int_t^{t + \kappa}\int_{\T^d\times \R^d}L(x,z)\psi_s(x,dz)m_s(dx)ds + \mathbb{E}_{\mathbb{P}'}[Z\cdot(X^\psi_{t+ \kappa}-X)]\\
& + \theta \kappa + (\kappa + \|X^\psi_{t + \kappa} - X_t\|_\infty)\omega(\kappa + \|X^\psi_{t + \kappa} - X_t\|_\infty) \bigg\},
\ea
 \]
where $\omega$ is a real continuous function such that $\omega(0) = 0$ which does not depend on $(\psi,m)$. We then obtain
\[
\ba
0 \geq \inf_{(\psi,m)} \bigg\{&\int_t^{t + \kappa}\mathbb{E}_{\mathbb{P}'}[\int_{ \R^d}L(X_s^\psi,z) + Z\cdot(z\psi_s(X^\psi_s,dz))]ds\\
& + \theta \kappa + (\kappa + \|X^\psi_{t + \kappa} - X_t\|_\infty)\omega(\kappa + \|X^\psi_{t + \kappa} - X_t\|_\infty) \bigg\}.
\ea
 \]
Using the definition of $H$, we then obtain
\[
\ba
0 \geq \inf_{(\psi,m)} \bigg\{&\int_t^{t + \kappa}\mathbb{E}_{\mathbb{P}'}[-H(X_s^\psi,Z)]ds  + \theta \kappa + (\kappa + \|X^\psi_{t + \kappa} - X_t\|_\infty)\omega(\kappa + \|X^\psi_{t + \kappa} - X_t\|_\infty) \bigg\}.
\ea
 \]
 Remark now that thanks to the bound on the support of $\psi$, we know that $\|X^\psi_{t + \kappa} - X_t\|_\infty \leq K \kappa$. Hence, dividing by $\kappa$ and taking the limit $\kappa \to 0$ yields the required result, because $\omega$ does not depend on $(\psi,m)$.
 
 We now come back to the case in which we may not be able to consider such a couple $(X,Z)$ as we did above. Note that there exists $(X,Z)$ whose law is given by $\mathcal{L}(X,Z) = µ(dx)\psi^*(x,dz)$, but possibly $X_t^\psi \ne X$. Consider a minimizing sequence $(\psi^n,m^n)$ of the infimum. For any $n$, $\eta_n > 0$, thanks to classical results such as Lemma 5.23 in \citep{carmona2017probabilistic}, there exists $(X_n,Z_n)$ which has the same law as $(X,Z)$ and which is such that $\|X_n - X^{\psi^n}_t\|_\infty \leq \eta_n$. To extend the previous computations, it then suffices to take a sequence $\eta_n$ converging fast toward $0$ and a sequence $\kappa_n$ converging sufficiently slowly toward $0$, and to take the limit $n \to \infty$.
\end{proof}
We now produce an argument which is classical in the theory of viscosity solutions, and usually crucial to the so-called Perron's method: the infimum of super-solution of the HJB equation is itself a super-solution of the HJB equation.
\begin{Prop}
Assume that \eqref{eq:hyppoly}, \eqref{hyp:dpl} and \eqref{hyp:dxl} hold. Then, the value function $U$ is a super-solution of \eqref{hjb}.
\end{Prop}
\begin{proof}
The result holds if it holds on $(0,T-\eps)$ for arbitrary $\eps > 0$, hence we only prove the result in this latter case. The rest of the proof follows the line of Proposition 4.3 in \citep{crandall1992user}. Consider $(t,µ) \in (0,T-\eps)\times \mptd$ and $(\theta,\psi^*) \in \partial^-U(t,µ)$. For the moment we omit the effect of the dependence in the time variable and come back on this question later on. 

There exists a smooth function $w : (0,+\infty)\to (0,+\infty)$ such that $\lim_{x \to 0}\omega(x) = 0$ and for any $µ' \in \mptd$, $\gamma$ coupling between $µ'$ and $µ$,
\be\label{eq:1690}
U(µ') - U(µ) - \int_{\T^d\times \T^d\times \R^d}z\cdot(y-x)\gamma(dx,dy)\psi^*(x,dz)   \geq -c(\gamma)\omega(c(\gamma)),
\ee
where $c(\gamma)^2 = \int_{\T^{2d}}|x-y|^2\gamma(dx,dy)$. We now define $\gamma_K$ as the minimum of the function
\[
\gamma \to U_K((\pi_2)_\#\gamma) - \int_{\T^d\times \T^d\times \R^d}z\cdot(y-x)\gamma(dx,dy)\psi^*(x,dz) + 2 c(\gamma)\omega(c(\gamma)).
\]
over $\gamma \in \mathcal{P}(\T^{2d})$ such that $(\pi_1)\#\gamma = \mu$. Note $µ_k = (\pi_2)_\#\gamma_K$. By definition of $\gamma_K$, we obtain
\[
U_K(µ_K) - \int_{\T^d\times \T^d\times \R^d}z\cdot(y-x)\gamma_K(dx,dy)\psi^*(x,dz) + 2 c(\gamma_K)\omega(c(\gamma_K)) \leq U_K(µ) - U(µ) + U(µ),
\]
for $U_K$ defined in \eqref{defUK}. Recalling \eqref{eq:1690}, we obtain
\[
\ba
U_K(µ_K) + 2 c(\gamma_K)\omega(c(\gamma_K)) \leq U_K(µ) - U(µ) + U(µ_K) + c(\gamma_K)\omega(c(\gamma_K)).
\ea
\]
Hence, we deduce that in the limit $K \to \infty$, since $U_K(µ) \to U(µ)$, that $µ_K$ converges toward $µ$. The main interest of the previous is that the optimality of $\gamma_K$ yields that we can consider an element in the sub-differential of $U_K$ at $µ_K$ which converges toward $\psi^*$ as $K \to \infty$. To explain this fact in an understandable fashion, we pass by the formalism of random variables, although it is not necessary. Consider $(\Omega',\mathcal{A}',\mathbb{P}')$ an atomeless probabilistic space. Define $\tilde U_K(Y) = U_K(\mathcal{L}(Y))$ for $Y$ a $\T^d$ valued random variable over $\Omega'$. Consider a couple $(X_K,Y_K,Z_K)$ whose law is given by $\gamma_K(dx,dy)\psi^*(x,dz)$. Denote also $\tilde c(X,Y) = c (\mathcal{L}(X,Y))$. Remark that we have already proven $\tilde c(X_K,Y_K) \to 0$ as $K \to \infty$. Then by construction of $\gamma_K$, it follows that 
\[
Z_K + 2\frac{X_K-Y_K}{\tilde c(X_K,Y_K)}(\omega(\tilde c (X_K- Y_K)) + \omega'(\tilde c(X_K,Y_K))\tilde c(X_K,Y_K)) \in \partial^-\tilde U_K(Y_K).
\]
Denoting the previous element by $Z_K + \eta_K$, we observe that $(\eta_K)_{K > 0}$ is a sequence of bounded random variables which converges uniformly toward $0$. 

We now observe that the presence of a time variable would not have perturbed the previous analysis. Hence, we assume that we are also given two sequences $(t_K)_K$ and $(\theta_K)_K$ valued in respectively $[0,T-\eps)$ and $\R$ such that
\[
(\theta_K,Z_K + \eta_K) \in \partial^-\tilde U_K(t_k,Y_K),
\]
with $\lim_{ K \to \infty} \theta_K = \theta$ and $\lim_{ K \to \infty} t_K = t$. Since $U_K$ is a viscosity supersolution of \eqref{hjb}, we deduce that
\[
-\theta_K + \mathbb{E}[H(Y_K, Z_K + \eta_K)] \geq 0.
\]
Passing to the limit $K \to \infty$ we finally obtain the required result.


\end{proof}

\subsection{Comparison with the more usual value functions}
In this section, we want to compare two value function of the deterministic problem. The first one being given by our reformulation of the cost and the second one in a more standard way, to show that they coincide. In order to avoid the technical problem of the terminal condition, we consider a smooth function $G : \mptd \to \R$ and consider 
\[
U(t,µ) = \inf_{\psi,m}\left\{\int_t^T\int_{\T^d\times \R^d}L(x,z)\psi_s(x,dz)m_s(dx)ds + G(m_T)\right\},
\]
where $t \in [t,T], µ \in \mptd$, and the infimum is taken over $\cup_{\nu \in \mptd} \mathcal{A}dm(t,µ,\nu)$.
The second function we consider is
\[
V(t,µ) = \inf_{\alpha,m}\left\{\int_t^T \int_{\T^d}L(x,\alpha_s(dx))m_s(dx)ds + G(m_T) \right\},
\]
where $t \in [t,T], µ \in \mptd$, and the infimum is taken over all pairs $(\alpha,m)$ such that $m \in \mathcal{C}([t,T],\mptd)$, $m_t = µ$ and $\alpha : [t,T]\times \T^d \to \R^d$ is measurable and for all $\varphi \in \mathcal{C}^1([t,T]\times \T^d,\R)$,
\[
\int_{\T^d}\varphi(T,x)m_T(dx) - \int_{\T^d}\varphi(t,x)\mu(dx) = \int_t^T\int_{\T^d}(\partial_t\varphi(t,x) + \alpha(t,x)\cdot \nabla_x \varphi(t,x))m_s(dx)ds.
\]
Hence, $V$ is the value function of the optimal control problem on the set of measures as it is more oftently defined.\\

We want to show that those two functions are equal. In order to do so, we are going to use the comparison principle Theorem \ref{thm:compgen} on the associated HJB equation
\be\label{hjb0}
-\partial_t U + \int_{\T^d}H(x,D_µU(t,x))µ(dx) = 0 \text{ in } (0,T)\times \mptd,
\ee
with terminal condition $G$.

Concerning the cost function $L$, we assume here that \eqref{eq:hyppoly}, \eqref{hyp:dpl} and \eqref{hyp:dxl} hold. We have already seen that $U$ is a viscosity solution of \eqref{hjb0}. We start by showing the following.

\begin{Prop}
The function $V$ is a viscosity super-solution of \eqref{hjb0}.
\end{Prop}
\begin{proof}
From an immediate adaptation of Gangbo and Tudorascu \citep{gangbo2019differentiability}, we know that for any $t < T, µ \in \mptd, (\theta,\xi) \in \partial^-_{clas}V(t,µ)$ such that $\xi$ is a bounded function,
\[
-\theta + \int_{\T^d}H(x,\xi(x))µ(dx) \geq 0.
\]
Hence taking $(\theta,\psi) \in \partial^-V(t,µ)$, we deduce that
\[
-\theta + \int_{\T^d}H\left(x,\int_{\R^d}z\psi(x,dz)\right)µ(dx) \geq 0,
\]
and the result follows from Jensen's inequality.
\end{proof}
 
 We now turn to the more subtle property of viscosity sub-solution. We are not able to establish it in all its generality, but nonetheless, the following result is sufficient to apply Theorem \ref{thm:compgen}. Indeed, the interested reader might have noted that in order to obtain comparison principle by means of our doubling of variable argument, we only need the viscosity properties to hold for elements of the sub-(or super) differential which are elements of the super-differential of the $2$-Wasserstein distance at some point.
 
 \begin{Prop}\label{prop:tt}
Assume that $D_pH$ is locally Lipschitz continuous and that for all $x\in \T^d, \xi,p \in \R^d$ with $\xi \ne 0$
\be\label{hyphprop}
\xi D_{pp}H(x,p)\xi > 0.
\ee
 Take $\lambda >0$ and consider $(\theta,\psi) \in \partial^+V(t,µ)$ such that $\psi$ satisfies $µ(dx)\psi(x,dz) = (\pi_1,\lambda(\pi_1 - \pi_2))_{\#}\gamma^o(dx,dz)$ for $\gamma^o\in \Pi^{opt}(µ,\nu)$ for some $\nu \in \mptd$. Then 
 \[
-\theta + \int_{\T^d\times \R^d}H(x,z)\psi(x,dz)µ(dx) \leq 0,
\]
 \end{Prop}
 \begin{proof}
Consider $(X,Y)$ a couple of random variables on a standard probability space such that $\mathcal{L}(X,Y) = \gamma^o$ and define
  \be\label{def:F}
F : \T^d\times\R^d \to \R^d, (x,p) \to -D_pH(x,\lambda p).
\ee
 As it is done in Proposition 5.30 in Santambrogio \citep{santambrogio2015optimal} for instance, we want to show that $m_s := \mathcal{L}(X + sF(X,X-Y))$ satisfies a continuity equation for a certain drift, and that evaluating the cost of this drift will yields the required inequality. \\
 
 Take $s > 0$ and $z \in Supp(m_s)$. By construction, there exists $(x,y) \in Supp(\gamma^o)$ such that $z = x + sF(x,x-y)$. We want to show that there exists at most one such couple $(x,y) \in \gamma^o$. Hence, take $(x,y), (x',y') \in Supp(\gamma^o)$ such that
\[
x + sF(x,x-y) = x' + sF(x',x'-y').
\]
Recall that $F$ is $C>0$ Lipschitz continuous thanks to the assumptions on $D_pH$. We deduce that 
\be\label{eqfirst}
|x - x'| = s|F(x',x'-y') - F(x, x-y)| \leq Cs (|x - x'| + |(x-y) - (x'-y')|).
\ee
Let us also write
\[
x - x' = s(F(x',x'-y') - F(x',x-y)) + s (F(x',x -y) - F(x,x-y)).
\]
Moreover, from \eqref{hyphprop}, we obtain that $-F$ is $\beta > 0$ monotone in $p$ for some $\beta > 0$, when $F$ is restricted to $\T^d\times [-1,1]^d$. Taking the scalar product against $(x'-y') - (x-y)$, we deduce that 
\[
\langle x - x' + s(F(x,x-y)-F(x',x-y)), (x'-y') - (x-y)\rangle \leq -\beta|(x'-y') - (x-y)|^2.
\]
Rearranging, we obtain that
\[
\beta|(x'-y') - (x-y)|^2 \leq |x - x'|^2 + Cs|x-x'|.|(x'-y') - (x-y)| - \langle x- x',y - y'\rangle.
\]
Since $(x,y), (x',y') \in Supp(\gamma^o)$, we obtain that the last scalar product is non-negative and thus that 
\[
\beta|(x'-y') - (x-y)|^2 \leq |x - x'|^2 + Cs|x-x'||(x'-y') - (x-y)|.
\]
Plugging this estimate into \eqref{eqfirst}, we finally deduce that
\[
|x - x'| \leq Cs\left(|x - x'| + Cs \left(\frac{Cs + \sqrt{C^2s^2 + 4\beta}}{2\beta}\right)|x -x'|\right).
\]
Hence for $s$ sufficiently small (compared to a constant which depends only on $C$ and $\beta$), we deduce that $2|x-x'| \leq |x-x'|$ and thus that $x = x'$ and hence that $y = y'$. This result of uniqueness has strong consequences. In particular, it allows to define two measurable maps $X_s(z)$ and $Y_s(z)$ such that for any $z \in Supp(m_s)$, $z = X_s(z) +sF(X_s(z), X_s(z) - Y_s(z))$. This implies in particular that $(\alpha,m)$ solve the continuity equation (at least in short time) with
\[
\alpha_s(z) = -D_pH(X_s(z),\lambda(X_s(z)- Y_s(z))).
\]
From this we deduce that 
\[
\ba
V(t,µ) &\leq \int_{t}^{t + \delta}\int_{\T^d} L(z,\alpha_s(z))m_s(dz)ds + V(t + \delta, m_{t + \delta})\\
& \leq \int_{t}^{t + \delta}\int_{\T^d} H(X_s(z),\lambda(X_s(z)- Y_s(z))) -\lambda\alpha_s(z)\cdot(X_s(z)-Y_s(z)) m_s(dz)ds\\
&\quad +\int_{t}^{t + \delta}\int_{\T^d} L(z,\alpha_s(z))- L(X_s(z),\alpha_s(z))m_s(dz)ds + V(t + \delta, m_{t + \delta}).\\
& \leq  \int_{t}^{t + \delta}\mathbb{E}[H(X,\lambda(X-Y)) + \lambda D_pH(X,\lambda(X-Y))\cdot(X-Y)]ds\\
&\quad + C\int_t^{t+\delta}\mathbb{E}[sD_pH(X,\lambda(X-Y))]ds + V(t + \delta, m_{t + \delta}).
\ea
\]
Hence, we deduce that
\[
\ba
0 \leq&  \int_{t}^{t + \delta}\mathbb{E}[H(X,\lambda(X-Y)) + \lambda D_pH(X,\lambda(X-Y))\cdot(X-Y)]ds\\
&\quad + C\int_t^{t+\delta}\mathbb{E}[sD_pH(X,\lambda(X-Y))]ds +\theta\delta - \delta\mathbb{E}[\lambda(X-Y)\cdot D_pH(X,\lambda(X-Y)] + o(\delta).
\ea
\]
We then deduce the result by diving by $\delta$ and taking the limit $\delta \to 0$.
 \end{proof}

 As a consequence of the two previous Propositions as well as of Theorem \ref{thm:compgen}, we obtain the following.
  \begin{Theorem}
 Assume that \eqref{eq:hyppoly}, \eqref{hyp:dpl}, \eqref{hyp:dxl} and \eqref{hyphprop} hold and that $D_pH$ is locally Lipschitz continuous, then the two functions $U$ and $V$ are equal.
 \end{Theorem}
 \begin{proof}
 The proof simply consists in using Theorem \ref{thm:compgen} to compare two time $U$ and $V$. Thanks to \eqref{eq:hyppoly}, \eqref{hyp:dpl} and \eqref{hyp:dxl}, we know that $U$ is continuous. The same argument as the one we presented for $U$ yields that $V$ is also continuous. Hence,  they both satisfy the terminal condition $G$ and the result is proven.
 \end{proof}

\section{Conclusion and perspectives}\label{sec:conclusion}
\subsection{Summary of the results}
In conclusion, we summarize the results that we have brought on the value function $U$ defined in \eqref{defUgood}.

We introduce the following Hypothesis.

\begin{hyp}\label{hyp:final}
The Hamiltonian $H$ is well defined and continuous and there exists $k > 1$ and $C \geq 0$ such that for all $x,y \in \T^d$, $\alpha,p \in \R^d$
\[
L(x,\alpha) \leq C( 1 + |\alpha|^k),
\]
\[
|D_xL(x,\alpha)| \leq C ( 1 + L(x,\alpha)),
\]
\[
|D_pL(x,\alpha)||\alpha|\leq C (1 + L(x,\alpha)),
\]
\[
|H(x,p) - H(y,p)| \leq C(1+|p|)|x-y|.
\]
\end{hyp}
Recall that we denote by $\omega(t) := \sup_{s\leq t,µ,\nu \in \T^d} U_{det}(s,µ,\nu)$.
\begin{Theorem}
Assume that Hypothesis \ref{hyp:final} holds and that the target process $(\nu_t)_{t \geq 0}$ is a jump process described by the operator $\mathcal{T}$ and the intensity $\lambda$. Assume that $\lambda$ is continuous and that there exist $C >0$ and $\gamma > -1$ such that for $T-t \leq C^{-1}$
\[
 \lambda(t)\omega(t) \leq C (T-t)^{\gamma},
\]
 and
\[
 \lambda(t) \leq C (T-t)^{-1}\int_t^T\lambda.
\]
Then $U$ defined in \eqref{defUgood} is the unique viscosity solution of \eqref{hjbl} such that 
\[
\lim_{t \to T}\sup_{µ,\nu}|U(t,µ,\nu) - U_{det}(t,µ,\nu)| = 0.
\]
\end{Theorem}

\begin{Theorem}
Assume that Hypothesis \ref{hyp:final} holds for $k = 2$, with the reverse inequality
\[
\forall x \in \T^d, \alpha \in \R^d, \quad C^{-1}|\alpha|^2 -C\leq L(x,\alpha),
\]
for some $C > 0$ and that the target process $(\nu_t)_{t \geq 0}$ is given by $\nu_t = (\tau_{W_t})_{\#}\nu$ where $(W_t)_{t \geq 0}$ is the strong solution of \eqref{sde}. Assume that $\sigma$ is bounded and satisfies
\[
\sigma(t) \sim K(T-t)^{\gamma} \text{ as } t \to T,
\]
for some $K \ne 0,\gamma > \frac 12$. Then $U$ is the unique viscosity solution of \eqref{hjbw} which satisfies 
\[
\lim_{t \to T}\sup_{µ,w}|U(t,µ,w) - U_{det}(t,µ,(\tau_w)_{\#}\nu)| = 0.
\]
\end{Theorem}
\begin{Rem}
In particular $L$ has exactly a quadratic growth here. As we mentioned in Remark \ref{rem:BDG}, the case in which it grows faster than quadratically can be treated similarly by using the Burkholder-Davis-Gundy inequalities in the study of the singularity.
\end{Rem}

 \subsection{Potential approximating schemes for the HJB equation}
Discretizing both time and the space of measures, one can arrive at usual discrete scheme for dynamic optimal control problem, namely using the notion of Wasserstein barycenters \citep{carlier}. It seems that the stability properties of viscosity solutions should be helpful to prove some convergence properties. On the other hand, fast methods to compute Wasserstein barycenters now exist \citep{cuturi2014fast} and could lead to a tractable numerical treatment of the problem.
 
 \subsection{More general optimal control problem}
 The techniques developed in Section \ref{sec:hjb} seem to be well-suited to study more general optimal control problems on the space of probability measures. With Pierre-Louis Lions (Coll\`ege de France), we are currently generalizing them to treat the case of the control of the parabolic continuity equation
 \[
 \partial_t m - \nu\Delta m + \text{div}(\alpha m) = 0 \text{ in } (0,T)\times \T^d,
 \]
 where the control is still $\alpha$, but the presence of the term in $\nu$ makes the analysis more complex.
 
\section*{Acknowledgments}
The author is thankful to P.-L. Lions and to S. Sorin for helpful discussions on this subject which helped, hopefully, in the presentation of the paper. The author acknowledges a partial support from the Chair FDD from Institut Louis Bachelier and from the Lagrange Mathematics and Computation Research Center. Data sharing not applicable to this article as no datasets were generated or analysed during the current study.
 
\bibliographystyle{plainnat}
\bibliography{bibmatrix}

\appendix

\end{document}